\documentclass[11pt]{article}
\usepackage[utf8]{inputenc}
\usepackage{amsmath,amssymb, dsfont}
\usepackage{amsthm}
\usepackage{hyperref}
\usepackage{authblk}

\usepackage[english]{babel}
\usepackage{amssymb,latexsym,amsfonts}
\usepackage{amsmath,amssymb}
\usepackage{fancyhdr}
\usepackage{tikz-cd} 
\usepackage{url}
\usepackage{longtable,hhline}
\usepackage{graphicx}
\usepackage{graphics,epsfig,color}
\graphicspath{ {images/} }
\usepackage{wrapfig}
\usepackage{eurosym}

\usepackage[T1]{fontenc}

\usepackage{amsbsy}

\newtheorem{thm}{\textsc{Theorem}}

\newtheorem{pro}{\textsc{Proposition}}
\newtheorem{cor}{\textsc{Corollary}}
\newtheorem{lem}{\textsc{Lemma}}
\newtheorem{rmq}{\textsc{Remark}}
\newtheorem{defn}{\textsc{Definition}}
\newtheorem{exemple}{\textsc{Example}}
\newtheorem{notation}{\textsc{Notation}}

\author{Frédéric PROTIN}
\affil{Past \& Future}

\title{A Framework for Supervised and Unsupervised Learning via Reproducing Kernel Hilbert Spaces}

\begin{document}

\maketitle




\begin{abstract}
We introduce a unified framework for supervised and unsupervised learning through density estimation using reproducing kernel Hilbert spaces (RKHSs). To this end, we introduce the notion of a nested RKHS, that is, an increasing sequence of RKHSs whose union is dense in an ambient $L^2$-space. This construction allows us to exploit the computational structure of RKHS methods at each finite approximation level while preserving the approximation power of the entire $L^2$-space.

We first study statistical learning within a fixed RKHS. In both unsupervised and supervised settings, we establish strong convergence results for empirical averages of random elements taking values in an RKHS continuously embedded in $L^2$. We then leverage the nested RKHS structure to combine these statistical convergence results with the density of the union of the RKHSs in the ambient $L^2$-space. This framework is applied to probability density estimation and binary Bayesian classification. 

For density estimation, we construct explicit empirical estimators of the orthogonal projections of the target density onto the RKHSs and prove their almost sure convergence. A diagonal selection procedure then yields a sequence of estimators converging almost surely to the target density in $L^2$. The resulting estimators are expressed explicitly in terms of kernel sections evaluated at the observations, thereby avoiding the numerical minimization of empirical objective functions.

For binary Bayesian classification, we formulate the approximation of the Bayes target as the minimization of suitable population and empirical risk functionals. We prove the consistency of the empirical minimizers and leverage the nested RKHS structure to establish convergence towards the Bayes target in $L^2$. We further provide a geometric interpretation in terms of hyperplanes and weighted distances in the ambient $L^2$-space, revealing a connection with the feature-space geometry of kernel Support Vector Machines. Finally, we show that the population risk admits a direct probabilistic interpretation as the probability of misclassification. The resulting framework thus connects statistical learning, RKHS approximation, geometric classification, and Bayesian decision theory within a common nested RKHS setting.
\end{abstract}

\section{Introduction}

Estimating an unknown probability density from a random sample is a fundamental problem in statistics and machine learning. Classical nonparametric approaches, such as kernel density estimation, provide flexible estimators without imposing a parametric model on the underlying distribution. Reproducing kernel Hilbert spaces (RKHSs) provide a natural functional framework for nonparametric estimation, in which functions can be approximated by finite linear combinations of kernel sections. Kernel methods have consequently been widely used in density estimation, classification, and more generally in statistical learning \cite{sang2023, pereverzyev2022, scholkopf2002learning, vapnik1998statistical}.

RKHS-based approaches have also been considered specifically for probability density estimation. Vapnik and Mukherjee proposed a Support Vector Method for multivariate density estimation, in which density estimation is formulated as an inverse problem and solved using kernel-based methods \cite{vapnik1999support}. More recently, Kazashi and Nobile developed a density estimation method in an RKHS framework, with the density approximated by finite linear combinations of kernel sections, and established convergence results for the resulting estimators \cite{kazashi2023density}. RKHS-based density approximation has also been considered in the context of functional data. Saeidi, Aminghafari, and Ashkartizabi proposed Fdmclust, a model-based clustering method for functional data based on probability density approximation in an RKHS framework \cite{saeidi2025fdmclust}. Their approach uses projections onto a Mercer kernel and combines these projections with PCA or ICA to handle Gaussian and non-Gaussian functional data, respectively, with parameter estimation performed by an EM algorithm. These works illustrate the versatility of RKHS-based density approximation and the importance of the choice of the underlying function space and approximation scheme.

Existing RKHS-based density estimation methods generally work within a prescribed RKHS, whereas our approach considers an increasing family of RKHSs whose union is dense in the ambient $L^2$-space. This distinction is central to the framework developed in this paper. Rather than requiring a single RKHS to provide a sufficiently rich approximation space, we work with a sequence of RKHSs
\[
H_0\subset H_1\subset\cdots\subset L^2(\mathcal X,\beta)
\]
whose union is dense in $L^2(\mathcal X,\beta)$. We refer to such a family as a \emph{nested RKHS}. After introducing the notion of $L^2$-RKHS in Subsection \ref{sub:RKHS}, we introduce and study nested RKHSs in Subsection \ref{sub:Nested_RKHS}. We then define in Subsection \ref{sub:Distinguished_nested_RKHS} a particular class of nested RKHSs, naturally associated with orthonormal bases of continuous functions of the $L^2$-space under consideration. This construction provides, in particular, a natural way to establish the existence of nested RKHSs. At each level, the problem remains within an RKHS and therefore admits a kernel representation, while increasing the level allows the approximation to converge to arbitrary elements of the ambient $L^2$-space.

The statistical basis of our approach is developed in Section \ref{subsec:statistical}. We consider random elements taking values in an RKHS continuously embedded in $L^2(\mathcal X,\beta)$ and study their expectations and empirical averages. In Subsection \ref{sec:random}, we establish the relevant convergence result in the unsupervised setting, while Subsection \ref{sec:framework} treats the corresponding supervised setting. These results provide the statistical foundation for the constructions developed subsequently.

We then apply the nested RKHS framework to Bayesian machine learning in Section \ref{sec:bayesian}. Our objective is to extend the fixed-RKHS results of Section \ref{subsec:statistical} so as to recover the target functions $f_X$ and $f_m$, rather than only their projections onto a given function space. We first formulate the corresponding approximation problems as the minimization of suitable risk functionals in Subsection \ref{sub:risk}. For each approximation level, we study the consistency of the associated empirical minimization procedure. The nested structure is then used in Subsection \ref{sub:minL2} to combine the approximation and statistical limits through a diagonal selection procedure. This yields a sequence of empirical estimators converging almost surely to the target function in $L^2(\mathcal X,\beta)$. In particular, the construction applies to probability density estimation and to the estimation of the Bayes target in binary classification. An important feature is that the resulting empirical minimizers admit explicit expressions in terms of kernel sections evaluated at the observations, so that their computation does not require the numerical minimization of a potentially costly empirical objective function.

In the classification setting, we restrict our attention to binary classification, which contains the essential structure of the Bayesian decision problem considered here. The relevant target is the difference between the two class-conditional densities weighted by their prior probabilities, equivalently represented through the regression function associated with the binary response. The resulting classification rule is determined by the sign of this target function. The convergence results obtained in Subsection \ref{sub:minL2} therefore provide a consistent approximation of the Bayes target through functions belonging to the successive RKHSs.

The proposed construction also admits a geometric interpretation, developed in Subsection \ref{lift}. The canonical embedding of $\mathcal X$ into $H_n$ allows a decision boundary in $\mathcal X$ to be viewed as a hyperplane in $H_n$. The Bayes decision function and its RKHS approximations can consequently be interpreted as vectors in the ambient $L^2$-space, while the empirical cost can be expressed in terms of weighted distances to the corresponding hyperplane. This establishes a geometric connection with kernel Support Vector Machines, as both approaches formulate classification in terms of separating hyperplanes in a feature space \cite{vapnik1998statistical,scholkopf2002learning}, albeit with different criteria.

Finally, in Subsection \ref{sub:probabilistic}, we return to the binary classification setting and give a probabilistic interpretation of the risk functional. We show that the population risk can be expressed directly in terms of the probability of misclassification. Thus, the minimization of the population risk is equivalent to the minimization of the probability of an incorrect classification. The Bayes classifier consequently admits three complementary interpretations: as the solution of a statistical risk minimization problem, as a distinguished direction in the $L^2$ geometry, and as the classifier minimizing the probability of misclassification.

The resulting framework provides a unified setting for supervised and unsupervised learning based on nested RKHSs. Its main feature is the combination of the statistical and computational structure of RKHS methods at each finite approximation level with the density of the nested family in the ambient $L^2$-space. This makes it possible to obtain explicit kernel-based estimators while retaining, in the asymptotic limit, access to the full target function rather than only to its projection onto a prescribed RKHS.

\section{RKHS framework}

Kernel methods in machine learning amount to working in a proper subspace of an $L^2$-space endowed with the structure of an RKHS. In order to extend this approach to the whole $L^2$-space, we introduce the notion of nested RKHSs, namely an increasing family of RKHSs whose union is dense in the ambient $L^2$-space. After introducing the notion of $L^2$-RKHS in Subsection \ref{sub:RKHS}, which is closely related to the classical concept of RKHS and is suitable for our purpose, we introduce and study nested RKHSs in Subsection \ref{sub:Nested_RKHS}. We then define in Subsection \ref{sub:Distinguished_nested_RKHS} a particular class of nested RKHSs, naturally associated with orthonormal bases of continuous functions of the $L^2$-space under consideration. This construction also provides a natural way to establish the existence of nested RKHSs.

\subsection{Reproducing kernel Hilbert spaces}\label{sub:RKHS}

In this subsection, we recall some basic facts about reproducing kernel Hilbert spaces and introduce the notion of $L^2$-RKHS, which will be useful throughout the paper. We refer, for instance, to \cite{rkhs} and \cite{ber} for further developments on RKHS.

\begin{defn}\label{def:rkhs-classical}
Let \(\mathcal{X}\) be a topological space. A \emph{reproducing kernel Hilbert space} (RKHS) with reproducing kernel \(K : \mathcal{X} \times \mathcal{X} \to \mathbb{R}\) is a Hilbert space \(H\) of functions \(f : \mathcal{X} \to \mathbb{R}\) such that:
\begin{itemize}
  \item For every \(x \in \mathcal{X}\), the function \(K(x, \cdot)\) belongs to \(H\).
  \item For every \(f \in H\) and every \(x \in \mathcal{X}\), the reproducing property holds:
  \[
  \langle f, K(x, \cdot) \rangle_H = f(x).
  \]
\end{itemize}
\end{defn}

\begin{rmq}
Let $x\in \mathcal{X}$. The function $K(x,\cdot)$ has the smallest norm among all functions in $H$ that take the same value at $x$. Indeed, by the Cauchy--Schwarz inequality, $$\displaystyle |f(x)|=|\langle f,K(x,\cdot)\rangle_H|\leq \|f\|_H\,K(x,x)^{1/2}$$ for all $f\in H$, with equality if and only if $f$ is proportional to $K(x,\cdot)$; hence $$\displaystyle \min_{f\in H,\ f(x)=K(x,x)}\|f\|_H=K(x,x)^{1/2}=\|K(x,\cdot)\|_H$$ (see also Theorem 6 in \cite{ber}).
\end{rmq}

\noindent For our purpose, we use a slightly different definition:

\begin{defn}\label{def:rkhs_}
Let \(\mathcal{X}\) be a separable metric space endowed with its Borel $\sigma$-algebra and a probability measure $\beta$. An \textbf{\(L^2\)-RKHS with reproducing kernel \(K :\mathcal{X} \times \mathcal{X}\rightarrow \mathbb{R}\)}, measurable with respect to the product Borel \(\sigma\)-algebra, is a Hilbert space \(H\) consisting of functions defined on \(\mathcal{X}\), satisfying the following properties:
\begin{itemize}
    \item There exists a linear continuous injection 
    \[
    \iota : H \to L^2(\mathcal{X}, \beta)
    \]
    that allows us to identify each element of \(H\) with a function in \(L^2(\mathcal{X}, \beta)\).
    \item The inner product on \(H\) is induced by \(L^2(\mathcal{X}, \beta)\) endowed with the canonical inner product via \(\iota\), i.e., for all \(f, g \in H\),
    \[
    \displaystyle\left\langle f, g \right\rangle_H = \left\langle \iota(f), \iota(g) \right\rangle_{L^2(\mathcal{X},\beta)}.
    \]
    \item For every \(x \in \mathcal{X}\), the function \(K(x, \cdot)\) belongs to \(L^2(\mathcal{X}, \beta)\).
    \item The reproducing property holds: for every \(f \in H\) and every \(x \in \mathcal{X}\),
\begin{equation}\label{eq:reproducing}
    \displaystyle\left\langle f, K(x, \cdot) \right\rangle_{L^2(\mathcal{X}, \beta)} = f(x).
\end{equation}
    \item Moreover, for all $x,y\in \mathcal{X},$ $$\displaystyle\left\langle K(x,.),K(y,.)\right\rangle_{L^2(\mathcal{X},\beta)}=K(x,y).$$
\end{itemize}
\end{defn}

\begin{rmq}
The image $\iota(H)$ is necessarily closed in $L^2(\mathcal{X},\beta)$, since $\iota$ is an isometry from the complete space $H$. Unlike in the classical definition of an RKHS, we do not require $K(x,\cdot)\in H$ for every $x\in\mathcal{X}$.
\end{rmq}

Let us now state the following property, showing that any $L^2$-RKHS is contained in an RKHS in the classical sense of Definition \ref{def:rkhs-classical}. Recall that, by Definition \ref{def:rkhs_}, the elements of $H$ are understood as functions with specified representatives, so that the pointwise evaluation $f(x)$ is well defined.

\begin{thm}\label{maxkernel}
Let $K:\mathcal{X}\times\mathcal{X}\rightarrow\mathbb{R}$ be such that $K(x,\cdot)\in L^2(\mathcal{X},\beta)$ for all $x\in\mathcal{X}$, and suppose that, for all $x,y\in\mathcal{X}$,

$$
\left\langle K(x,\cdot),K(y,\cdot)\right\rangle_{L^2(\mathcal{X},\beta)}=K(x,y).
$$

\noindent Set

$$
H_0:=\overline{\operatorname{span}\left(\bigcup_{x\in\mathcal X}K(x,\cdot)\right)}.
$$
Then $H_0$ is an RKHS, and is the greatest subspace of $L^2(\mathcal{X},\beta)$ for which $K$ satisfies the reproducing property \eqref{eq:reproducing}.
\end{thm}

\begin{proof}
First, let $H'\subset L^2(\mathcal X,\beta)$ be the greatest subspace for which $K$ satisfies the reproducing property \eqref{eq:reproducing}. Since $\overline{H'}$ also satisfies this property by continuity of the scalar product, we deduce that $H'$ is closed. By construction, $H_0\subset H'$. Let $l:H'\to\mathbb R$ be a continuous linear form that vanishes on $H_0$. By the Riesz representation theorem, there exists $h\in H'$ such that

$$
l=\langle\cdot,h\rangle_{L^2(\mathcal X,\beta)}.
$$
For a.e. $x\in\mathcal X$, since $K(x,\cdot)\in H_0$,

$$
h(x)
=\langle K(x,\cdot),h\rangle_{L^2(\mathcal X,\beta)}
=l(K(x,\cdot))
=0.
$$
Thus $h=0$, and hence $l=0$. Therefore, $H_0$ is dense in $H'$. Since $H_0$ is closed, we conclude that $H'=H_0$. Hence $H_0$ is the greatest subspace of $L^2(\mathcal X,\beta)$ for which $K$ satisfies the reproducing property.

\noindent Now, let
$$
V:=\operatorname{span}\{K(x,\cdot):x\in\mathcal X\}.
$$
For every $f\in V$ and every $x\in\mathcal X$, linearity and the assumed property of $K$ give

$$
\langle f,K(x,\cdot)\rangle_{L^2(\mathcal X,\beta)}=f(x).
$$

Since $V$ is dense in $H_0$, this property extends by continuity of the scalar product to every $f\in H_0$. Thus $H_0$ is an $L^2$-RKHS. Moreover, by definition, $K(x,\cdot)\in H_0$ for all $x\in\mathcal X$, so $H_0$ is an RKHS in the sense of Definition \ref{def:rkhs-classical}.
\end{proof}

Note that if the kernel $K:\mathcal{X}\times\mathcal X\rightarrow\mathbb{R}$ is continuous, then the canonical feature map $x\mapsto K(\cdot,x)$ is continuous in $H$. Hence, by the reproducing property, every $f\in H$ is continuous (see, e.g., \cite{ber}, Ch.~1).

\subsection{Nested RKHS}\label{sub:Nested_RKHS}

In this subsection, we introduce the concept of nested RKHSs, namely an increasing family of RKHS subspaces whose union is dense in an $L^2$-space, and establish some of its basic properties.

\begin{defn}\label{def:maximality}
With the notations of Definition \ref{def:rkhs_}, an $L^2$-RKHS $L\subset L^2(\mathcal{X}, \beta)$ with reproducing kernel \(K :\mathcal{X} \times \mathcal{X}\rightarrow \mathbb{R}\) is said {\textbf{maximal}} if $$\displaystyle L=\left\{u\in L^2(\mathcal{X}, \beta): \forall x\in \mathcal{X},u(x)=\left\langle u, K(x,\cdot)\right\rangle\right\}.$$
\end{defn}

 \begin{rmq}\label{lem:maximal}
 An $L^2$-RKHS $L\subset L^2(\mathcal{X}, \beta)$ with reproducing kernel
$K$ is maximal
if and only if $K(x,\cdot)\in L$ for all $x\in\mathcal{X}$, i.e. if it is an RKHS in the sense of Definition \ref{def:rkhs-classical}.
Indeed, if $K(x,\cdot)\in L$ for all $x\in\mathcal{X}$, then $L$ is
maximal by Theorem \ref{maxkernel}. Conversely, suppose that $L$ is
maximal and fix $x\in\mathcal{X}$. Then
$K(x,y)=\langle K(x,\cdot),K(y,\cdot)\rangle$ for all $y\in\mathcal{X}$
implies $K(x,\cdot)\in L$ according to Definition \ref{def:maximality}.
Thus, within $L^2(\mathcal X,\beta)$, maximal $L^2$-RKHSs coincide
with classical RKHSs.

More generally, let $H$ be an $L^2$-RKHS with reproducing kernel $K$. Then $H$ is an RKHS in the sense of Definition \ref{def:rkhs-classical}
with reproducing kernel $\displaystyle \widetilde K(x,\cdot):=P^{L^2(\mathcal X,\beta)}_H K(x,\cdot),$ $ x\in\mathcal X$.
Indeed, for every $f\in H$ and every $x\in\mathcal X$,
\[\displaystyle\left\langle f,\widetilde K(x,\cdot)\right\rangle_H=\left\langle f,K(x,\cdot)\right\rangle_{L^2(\mathcal X,\beta)}=f(x).\]
\end{rmq}

\begin{defn}\label{def:rkhs}
Let $\mathcal{X} \subset \mathbb{R}^n$ be a subset, and consider some Hilbert space $H\subset L^2(\mathcal{X}, \beta)$ endowed with the canonical scalar product for some probability measure $\beta$ on $\mathcal{X}$. Then $H$ is said to be a {\it{nested RKHS}} if there exists an
increasing sequence $H_0\subset H_1\subset\ldots\subset H$ of
RKHS, each endowed with a reproducing kernel for the scalar
product induced by that of $L^2(\mathcal{X},\beta)$, such that
$\displaystyle\overline{\bigcup_n H_n}=H$. \\
Moreover if $dim(H_n)<+\infty$ for all $n\in\mathbb{N}$, then $H$ is said to be a {\it{finitely generated nested RKHS}}.
\end{defn}

Note that a nested RKHS is not an RKHS in general. In the sequel, we suppose (without loss of generality) $H_n\neq H_{n+1}$ for all $n\in \mathbb{N}$. 

\begin{notation}
With the notations of Definition \ref{def:rkhs}, if $K_n:\mathcal{X}\times \mathcal{X}\rightarrow\mathbb{R}$ denotes the reproducing kernel of the RKHS $H_n$, we say that $\displaystyle \left(H, (H_n)_{n\in\mathbb{N}},(K_n)_{n\in\mathbb{N}}\right)$ forms a nested RKHS in $L^2(\mathcal{X},\beta)$. 
\end{notation}

If $H$ is a Hilbert space and $H'\subset H$ is a closed subspace, we denote by $P^{H}_{H'}$ the orthogonal projection of $H$ onto $H'$. Let us isolate a property that we will use often:

\begin{rmq}
 Fix $x_0\in \mathcal{X}$ arbitrarily. Then for all $y\in \mathcal{X}$, for $0\leq n\leq m$ and for $n\leq k$, \begin{equation}\label{eq:perp}K_n(y,\cdot)\perp K_m(x_0,\cdot)-K_k(x_0,\cdot).\end{equation}
 Indeed $\displaystyle\langle K_n(y,\cdot)\mid K_m(x_0,\cdot)\rangle =K_n(y,x_0)$ for all $y\in \mathcal{X}$, since $K_n(y,\cdot)\in H_n\subset H_m$, so $$\displaystyle\left\langle K_n(y,\cdot)\mid  K_m(x_0,\cdot)-K_k(x_0,\cdot) \right\rangle_H =0.$$  
\end{rmq}

 For $f\in L^2(\mathcal{X},\beta)$, define $\displaystyle (K_nf)(x):=\langle K_n(x,\cdot),f\rangle_{L^2(\mathcal{X},\beta)}.$ Now let us state a useful result.

\begin{thm}\label{dirac} Let $\mathcal{X} \subset \mathbb{R}^d$ be a subset, and let $\beta$ be a probability measure on $\mathcal{X}$. Consider a nested RKHS $(H, (H_n)_{n\in\mathbb{N}},(K_n)_{n\in\mathbb{N}})$ in $L^2(\mathcal{X}, \beta)$. Then :\\

0) For all $n\in\mathbb{N},$ $\displaystyle H_n=\overline{\mathrm{span}\{K_n(x,\cdot), x \in \mathcal{X}}\}.$\\

1) For all $x\in \mathcal{X}$, the $K_n(x,\cdot)$ is compatible with the orthogonal projections, in the sense that $\forall k,m,n\in \mathbb{N}$ such that $m,n\leq k$, $\displaystyle P^{H_k}_{H_n}\left(K_m(x,\cdot)\right)=K_{\min(n,m)}(x,\cdot)$.\\

2) The orthogonal projection $P^{L^2(\mathcal{X},\beta)}_{H_n}: L^2(\mathcal{X},\beta)\rightarrow H_n$ is expressed for $n\in\mathbb{N}$, $f\in L^2(\mathcal{X}, \beta)$ and $x\in \mathcal{X}$ by $\displaystyle (P^{L^2(\mathcal{X},\beta)}_{H_n}f)(x)=\langle K_n(x,\cdot),  f\rangle.$ \\

3) For all $x\in\mathcal{X}$, $\displaystyle  K_n(x, \cdot)\underset{{n\rightarrow +\infty}}{\rightarrow} K(x,\cdot)$ in norm if $H$ has a reproducing Kernel $K(\cdot,\cdot)$.\\

4)  If $H=L^2(\mathcal{X},\beta)$, then for all $f\in L^2(\mathcal{X},\beta)$, $\displaystyle  K_n f\underset{{n\rightarrow +\infty}}{\rightarrow} f$ in $L^2(\mathcal{X},\beta)$.\\

\end{thm}

\begin{proof}

\noindent 0) This is a direct consequence of Theorem \ref{maxkernel}.\\

\noindent 1) Let $k\geq m\geq n$. Since $H_n\subset H_k$, write $H_k= H_n\oplus (H_k\cap H_n^\perp)$. Then consider the decomposition $K_m(x,\cdot) = f_1+ f_2$, $f_1\in H_n$, $f_2\in H_k\cap H_n^\perp$. Thanks to \eqref{eq:perp}, $f_1=K_n(x,\cdot)$ and $f_2=K_m(x,\cdot)-K_n(x,\cdot)$ for all $x\in \mathcal{X}$. In other words, $P^{H_k}_{H_n}K_m(x,\cdot)=K_n(x,\cdot),$ which proves the claim. Now suppose $k\geq n\geq m$. Then $K_m(x,\cdot)\in H_m\subset H_n\subset H_k$, and we get $P^{H_k}_{H_n}K_m(x,\cdot)=K_m(x,\cdot),$ as claimed. \\

\noindent 2) We first show that $K_nf\in H_n$. Let $R_f\in H_n$ be the Riesz representer of the continuous linear functional
   $
   \displaystyle g\longmapsto \left\langle f,g\right\rangle_{L^2(\mathcal{X},\beta)}
   $
   on $H_n$. For every $x\in\mathcal X$, the reproducing property gives
   $$
   R_f(x)=\left\langle R_f,K_n(x,\cdot)\right\rangle_{L^2(\mathcal{X},\beta)}
   =\langle f,K_n(x,\cdot)\rangle_{L^2(\mathcal{X},\beta)}
   =(K_nf)(x).
   $$
   Hence $K_nf=R_f\in H_n$. Moreover, for every $g\in H_n$,

$$
\left\langle f-K_nf,g\right\rangle_{L^2(\mathcal{X},\beta)}=0
$$
so $f-K_nf\perp H_n$. Since $K_nf\in H_n$, it follows that $K_nf=P^{{L^2(\mathcal{X},\beta)}}_{H_n}f$.\\

\noindent 3) Since $\displaystyle\overline{\bigcup_n H_n}=H$, the orthogonal projections $P^H_{H_n}$ converge strongly to the identity. In other words, for $f\in H$, \begin{equation}\label{eq:knconvergence}
    \displaystyle\left\|P^H_{H_n} f - f\right\|_H\underset{n\rightarrow+\infty}{\longrightarrow} 0.
    \end{equation}
On the other hand, note that \begin{equation}\label{eq:PkequalKn}P^H_{H_n}K(x,\cdot)=K_n(x,\cdot).\end{equation} Indeed, for all $f\in H_n$, 
$$\displaystyle\langle P^H_{H_n}K(x,\cdot),f\rangle_H=\langle K(x,\cdot),f\rangle_H=f(x)=\langle K_n(x,\cdot),f\rangle_H.$$Since $P^H_{H_n}K(x,\cdot)\in H_n$ and $K_n(x,\cdot)\in H_n$, we obtain \eqref{eq:PkequalKn}. Substituting this equality in \eqref{eq:knconvergence} for $f=K(x,\cdot)$, we get $\|K_n(x,\cdot)-K(x,\cdot)\|_H\rightarrow 0.$\\

\noindent 4) $K_n=P^{L^2(\mathcal{X},\beta)}_{H_n}$ by point 2), and since $\displaystyle\overline{\bigcup_n H_n}=L^2(\mathcal{X},\beta)$ by hypothesis, the orthogonal projections $P^{L^2(\mathcal{X},\beta)}_{H_n}$ converge strongly to the identity. Hence $\displaystyle\left\|K_nf-f\right\|_{L^2(\mathcal{X},\beta)}\rightarrow 0$ for $f\in L^2(\mathcal{X},\beta).$

 \end{proof}


\subsection{Distinguished nested RKHS}\label{sub:Distinguished_nested_RKHS}

In this subsection, we define a particular class of nested RKHSs and prove that these spaces are naturally in bijection with orthonormal bases of continuous functions of the Hilbert space under consideration. In particular, this construction provides a class of nested RKHSs and establishes their existence.

\begin{defn}\label{def:nested}
A nested RKHS $(H, (H_n)_{n\in\mathbb{N}},(K_n)_{n\in\mathbb{N}})$ will be said {\it{distinguished}} if $(H_0,K_0)=(\{0\},0)$, and if for all RKHS $L\subset L^2(\mathcal{X}, \beta)$ with kernel $K\in \mathcal{C}(\mathcal{X}\times \mathcal{X}, \mathbb{R})$, $H_n\subset L\subset H_{n+1}$ for some $n\in\mathbb{N}$ implies $H_n=L$ and $K=K_n$, or $H_{n+1}=L$ and $K=K_{n+1}$. 
\end{defn}

In fact, we will establish that all closed subspaces of $L^2(\mathcal{X},\beta)$ can be endowed with a structure of nested RKHS. Let us first state the following fact:

\begin{lem}\label{lem:maximalBasis}
Let $(T_0,\ldots,T_n)$ be an orthonormal family in $L^2(\mathcal{X},\beta)\cap\mathcal{C}(\mathcal{X},\mathbb{R}).$ Then $\mathrm{span}(T_0,\ldots,T_n)$ is an RKHS for the reproducing kernel $\displaystyle (x,y)\mapsto \sum_{i=0}^n T_i(x)T_i(y).$
\end{lem}

\begin{proof}
A straightforward computation shows that $(x,y)\mapsto \displaystyle \sum_{i=0}^nT_i(x)T_i(y)$ is a reproducing kernel with respect to the canonical scalar product of $L^2(\mathcal{X},\beta)$ for $\mathrm{span}(T_0,\ldots,T_n)$. But, for $u\in L^2(\mathcal{X},\beta),$ $$\displaystyle u(x)=\left\langle u,\sum_{i=0}^nT_i(x)T_i\right\rangle$$ for all $x\in\mathcal{X}$ if and only if $u\in \mathrm{span}(T_0,\ldots,T_n)$. Thus $$\displaystyle\left(\mathrm{span}\left(T_0,\ldots,T_n\right),x\mapsto\sum_{i=0}^nT_i(x)T_i\right)$$ is a maximal $L^2$-RKHS, i.e. an RKHS according to Remark \ref{lem:maximal}.
\end{proof}

\begin{lem}\label{lem:distinguished-codimension}
Let $(H,(H_n)_{n\in\mathbb N},(K_n)_{n\in\mathbb N})$ be a distinguished nested RKHS such that each $H_n$ is finite-dimensional and each reproducing kernel $K_n$ is continuous on $\mathcal X\times\mathcal X$. Then $\displaystyle\dim(H_{n+1}/H_n)=1$ for every $n\in\mathbb N.$
\end{lem}

\begin{proof}
Fix $n\in\mathbb N$. Since $H_n\subset H_{n+1}$ and both spaces are finite-dimensional, suppose for a contradiction that $\dim(H_{n+1}/H_n)\geq 2.$ Choose a nonzero function $e\in H_{n+1}\cap H_n^\perp$, and set $\displaystyle L:=H_n\oplus\operatorname{span}\{e\}.$ Then $H_n\subsetneq L\subsetneq H_{n+1}.$ Let $(\varphi_1,\ldots,\varphi_d)$ be an orthonormal basis of $H_n$. Since $K_{n+1}$ is continuous, every element of $H_{n+1}$ is continuous. Thus the orthonormal family
\[
\left(\varphi_1,\ldots,\varphi_d,\frac{e}{\|e\|_{L^2(\mathcal X,\beta)}}\right)
\]
consists of continuous functions in $L^2(\mathcal X,\beta)$. By Lemma \ref{lem:maximalBasis}, $L$ is an RKHS with continuous reproducing kernel
\[
K_L(x,y)
=
\sum_{i=1}^d\varphi_i(x)\varphi_i(y)
+\frac{e(x)e(y)}{\|e\|_{L^2(\mathcal X,\beta)}^2}.
\]
The distinguished property therefore implies that $L=H_n$ or $L=H_{n+1}$, contradicting the strict inclusions above. Hence $\displaystyle\dim(H_{n+1}/H_n)=1.$

\end{proof}

\begin{thm}\label{th:distinguished}
Let $H\subset L^2(\mathcal{X},\beta)$, and let
$(H,(H_n)_{n\in\mathbb{N}},(K_n)_{n\in\mathbb{N}})$ be a nested RKHS such that $\displaystyle H_0=\{0\}$,
each $H_n$ is finite-dimensional, and each $K_n$ is continuous on
$\mathcal{X}\times\mathcal{X}$. Then the nested RKHS is distinguished if and only if $\displaystyle\dim(H_{n+1}/H_n)=1$ for every $n\in\mathbb{N}.$\\

Equivalently, it is distinguished if and only if there exists an orthonormal basis
$(T_n)_{n\in\mathbb{N}}$ of continuous functions in $H$ such that
\[
H_{n+1}=\operatorname{span}(T_0,\ldots,T_n),
\qquad
K_{n+1}(x,y)=\sum_{i=0}^nT_i(x)T_i(y).
\]
\end{thm}

\begin{proof}
Let $(T_k)_{k\in\mathbb{N}}$ be an orthonormal basis of continuous functions in $H$. Then
\[
K_{n+1}(x,y):=\sum_{i=0}^nT_i(x)T_i(y)
\]
defines a reproducing kernel of $\displaystyle H_{n+1}:=\operatorname{span}(T_0,\ldots,T_n).$ By Lemma \ref{lem:maximalBasis}, the $L^2$-RKHSs $(H_{n+1},K_{n+1})$ are maximal. Since
$\dim(H_{n+2}/H_{n+1})=1$, every RKHS $L$ satisfying
$H_{n+1}\subset L\subset H_{n+2}$ is equal to $H_{n+1}$ or $H_{n+2}$. Thus the nested RKHS is distinguished.\\

\noindent Conversely, let $(H,(H_i)_{i\in\mathbb{N}},(K_i)_{i\in\mathbb{N}})$ be a distinguished nested RKHS. By Lemma \ref{lem:distinguished-codimension},
\[
\dim(H_{i+1}/H_i)=1
\qquad\text{for every }i\in\mathbb{N}.
\]
In particular, $\dim(H_1)=1$. Fix $x_0\in\mathcal{X}$ such that $K_1(x_0,\cdot)\neq 0$, and, for each $i\in\mathbb{N}$, choose $x_i\in\mathcal{X}$ such that
\[
K_{i+1}(x_i,\cdot)-K_i(x_i,\cdot)\neq 0.
\]
Such choices are possible since $H_1\neq\{0\}$ and $H_i\neq H_{i+1}$, respectively. Thanks to \eqref{eq:perp}, we can construct an orthonormal system in $H$ by setting
\begin{equation}\label{eq:t0}
T_0:=\frac{K_1(x_0,\cdot)}
{\|K_1(x_0,\cdot)\|_{L^2(\mathcal{X},\beta)}}
\end{equation}
and, for every $k\in\mathbb{N}$,
\begin{equation}\label{eq:tkplusone}
T_{k+1}:=
\frac{K_{k+2}(x_{k+1},\cdot)-K_{k+1}(x_{k+1},\cdot)}
{\|K_{k+2}(x_{k+1},\cdot)-K_{k+1}(x_{k+1},\cdot)\|_{L^2(\mathcal{X},\beta)}}.
\end{equation}
By construction, $T_0\in H_1$ and, for every $k\in\mathbb{N}$, $\displaystyle T_k\in H_{k+1}\cap H_k^\perp$ and $\displaystyle T_k\neq 0.$ Since $\dim(H_1)=1$, we have $H_1=\operatorname{span}(T_0)$. Moreover, the codimension-one property gives
\[
H_{k+1}=H_k\oplus\operatorname{span}(T_k)
\qquad\text{for every }k\in\mathbb{N}.
\]
It follows by induction that $\displaystyle H_{n+1}=\operatorname{span}(T_0,\ldots,T_n)$ for every $n\in\mathbb{N}$. By Lemma \ref{lem:maximalBasis}, the reproducing kernel of
$\operatorname{span}(T_0,\ldots,T_n)$ is
\[
D_n(x,y):=\sum_{i=0}^nT_i(x)T_i(y).
\]
By uniqueness of the reproducing kernel, $K_{n+1}=D_n$. Finally, since
$\displaystyle\bigcup_{n\in\mathbb{N}}H_n$ is dense in $H$, the sequence $(T_n)_{n\in\mathbb{N}}$ is an orthonormal basis of $H$. This completes the proof.
\end{proof}

\begin{rmq}
We deduce from the proof of Theorem \ref{th:distinguished} that $\displaystyle K_{n+1}(x,y)=\sum_{i=0}^nT_i(x)T_i(y),$ where the functions $T_i$ are defined by \eqref{eq:tkplusone}, independently of the choice of $(x_i)_{i\in\mathbb{N}}\in\mathcal{X}^{\mathbb{N}}$.
\end{rmq}

\begin{rmq}
The operations
\[
A:(T_0,\ldots,T_n)\mapsto
\left(
(x,y)\mapsto\sum_{i=0}^kT_i(x)T_i(y)
\right)_{k=0,\ldots,n}
\]
and
\[
B:(K_1,\ldots,K_{n+1})\mapsto(T_0,\ldots,T_n),
\]
where \(T_0\) is defined by \eqref{eq:t0} and \(T_{k+1}\) by \eqref{eq:tkplusone} for \(k=0,\ldots,n-1\), are reciprocal, provided that the points \((x_k)\) are chosen as in the proof of Theorem \ref{th:distinguished}. Indeed, \(B\circ A=\operatorname{Id}\) follows by a direct computation, while \(A\circ B=\operatorname{Id}\) follows from the proof of Theorem \ref{th:distinguished}.
\end{rmq}

Once we have a distinguished nested RKHS $(H, (H_i)_i, (K_i)_i)$, given a strictly increasing map $k:\mathbb{N}\rightarrow\mathbb{N}$,
we still obtain a nested RKHS $(H, (H'_j)_j, (K'_j)_j)$ by setting
\[
H'_j:=H_{k(j)},\qquad K'_j:=K_{k(j)}.
\]
This operation will be implicitly used in the following examples.

\begin{exemple}
Let $\beta$ be a probability measure on $[-1,1]$ with infinite support. Let $(T_k)_{k\in\mathbb N}$ be an orthonormal basis of $L^2([-1,1],\beta)$ consisting of polynomials, with $\deg(T_k)=k$. Then $\displaystyle H_n=\mathcal P_n([-1,1])=\operatorname{span}(T_0,\ldots,T_n)$ has reproducing kernel $$\displaystyle K_n(x,y)=\sum_{k=0}^n T_k(x)T_k(y).$$ More generally, for $d\in\mathbb N^*$, an orthonormal basis of $\displaystyle H=L^2\left([-1,1]^d,\beta^{\otimes d}\right)$ is given by
\[
\left(
(x_1,\ldots,x_d)\mapsto\prod_{j=1}^d T_{k_j}(x_j)\right)_{(k_1,\ldots,k_d)\in\mathbb N^d}.
\]
Here $\displaystyle H_n=\mathcal P_n([-1,1])^{\otimes d}$ consists of polynomials of degree at most n in each variable. The reproducing kernel of $H_n$ is
\[
K_n(x,y)=\sum_{0\leq k_1,\ldots,k_d\leq n}
\prod_{j=1}^d T_{k_j}(x_j)T_{k_j}(y_j)=\prod_{j=1}^d
\left(\sum_{k=0}^n T_k(x_j)T_k(y_j)\right).
\]
Then $(H,(H_n)_n,(K_n)_n)$ is a nested RKHS. Indeed, Lemma \ref{lem:maximalBasis} ensures that each $H_n$ is maximal in the sense of Definition \ref{def:maximality}. 
\end{exemple}

\begin{exemple}
Let \(\beta\) be a probability measure on \(\mathcal X\subset\mathbb R^d\), with finite moments of all orders, whose support is not contained in any proper algebraic hypersurface. We also assume that polynomials are dense in \(L^2(\mathcal X,\beta)\), which holds, for instance, under the sufficient condition $$\displaystyle\int_\mathcal{X}e^{c\|x\|_{\mathbb{R}^d}}d\beta(x)<+\infty$$
for some $c>0$. Denote by $N_j\in\mathbb{N}$ the dimension of the space of homogeneous polynomials in $d$ variables of degree $j$. Let
$(T_{i,j})_{j\in\mathbb{N},1\leq i\leq N_j}$ be an orthonormal basis of the space of polynomials in $L^2(\mathcal{X},\beta)$, ordered according to their
degree, namely $\displaystyle \deg(T_{i,j})=j.$
\[
H_j:=\mathrm{span}\{T_{i,k}: 0\leq k\leq j, 1\leq i\leq N_k\},
\]
which is the space of polynomials of degree at most $j$. Then
$\displaystyle\left(
\overline{\bigcup_{j\geq 0}H_j},
(H_j)_{j\in\mathbb{N}},
(K_j)_{j\in\mathbb{N}}
\right)$
is a nested RKHS, where
\[
K_j(x,y)
=
\sum_{k=0}^j\sum_{i=1}^{N_k} T_{i,k}(x)T_{i,k}(y).
\]
\end{exemple}

\section{Statistical learning in RKHSs}\label{subsec:statistical}

In this section, we study random elements taking values in an RKHS continuously embedded in an \(L^2\) space, their expectations, and the associated strong law of large numbers, in both an unsupervised setting, considered in Subsection \(\ref{sec:random}\), and a supervised setting, considered in Subsection \(\ref{sec:framework}\).

Let \((\Omega,\mathcal T,\mu)\) be a probability space, let \(\mathcal X\subset\mathbb R^n\) be equipped with a Borel probability measure \(\beta\), and let \(X:\Omega\to\mathcal X\) be a random variable. Assume that its distribution \(\rho_X:=X_*\mu\) is absolutely continuous with respect to \(\beta\), with density \(f_X\in L^2(\mathcal X,\beta)\), i.e.
\[
\rho_X(B)=\int_B f_X\,d\beta
\]
for every Borel set \(B\subset\mathcal X\). We realize an i.i.d. sequence with the same distribution as \(X\) on the product probability space
\[
(\Omega^{\mathbb N},\mathcal T^{\otimes\mathbb N},
\mu^{\otimes\infty}),
\]
by setting
\[
X_i(\omega)=X(\omega_i),
\qquad
\omega=(\omega_1,\omega_2,\ldots)\in\Omega^{\mathbb N}.
\]
Let \(H\subset L^2(\mathcal X,\beta)\) be a reproducing kernel Hilbert space (RKHS) with kernel \(K\), with inner product induced by that of \(L^2(\mathcal X,\beta)\). We assume that the elements of \(H\) are functions that admit pointwise evaluation, and not merely equivalence classes in \(L^2\). We denote by \(\mathcal H:=L^2(\Omega,H,\mu)\).

Denote by \(i:\mathcal X\to H\) the canonical embedding, \(i(x)=K(x,\cdot)\). Assume that \(H\) is separable and that every \(f\in H\) is Borel measurable on \(\mathcal X\). By the reproducing property,
\[
\langle f,i(x)\rangle_H=f(x),\qquad f\in H,\ x\in\mathcal X.
\]
Thus, \(i\) is weakly measurable. By Pettis's measurability theorem, \(i\) is Borel measurable as a map from \(\mathcal X\) into \(H\). Consequently, \(i(X):\Omega\to H\) is measurable and, using the terminology of \cite{fpca}, p. 178, is a random element.

\noindent We also assume that
$
\displaystyle\int_\Omega K(X,X)^{1/2}\,d\mu<+\infty.
$ Since \(\|i(X)\|_H^2=K(X,X)\), and
\[
E_\mu(\|i(X)\|_H)
=\int_\Omega K(X,X)^{1/2}\,d\mu<+\infty,
\]
\(E_\mu(i(X))\) is well defined in the Bochner sense (Definition 7.2.1 in \cite{fpca}). 

\subsection{Unsupervised learning}\label{sec:random}

Let us recall the following definition.

\begin{defn}\label{def:almostsure}
Let $(a_n)_{n\in\mathbb N}$
be a sequence of $H$-valued random elements defined on
$(\Omega^{\mathbb N},\mathcal T^{\otimes\mathbb N},
\mu^{\otimes\infty})$, where $\mu^{\otimes\infty}$ denotes the product measure.
We say that $a_n$ converges almost surely towards $a\in H$ if
\[
\mu^{\otimes\infty}\left(
\left\{\omega\in\Omega^{\mathbb N}:
\lim_{n\to+\infty}\left\|a_n(\omega)-a\right\|_H=0
\right\}
\right)=1.
\]
We write
\[
a_n\underset{a.s.}{\longrightarrow}a.
\]
\end{defn}

\begin{pro}\label{pro:estimator1}
Under the above assumptions, we have
\[
E_\mu(i(X))=P_H^{L^2(\mathcal X,\beta)}f_X.
\]
Moreover, for the i.i.d. sequence \((X_i)_{i\in\mathbb N}\) defined above,
\[
\lim_{N\to+\infty}\frac1N\sum_{i=1}^N K(X_i,\cdot)
\underset{a.s.}{=}
P_H^{L^2(\mathcal X,\beta)}f_X.
\]
\end{pro}

\begin{proof}
The map $i(X):\Omega\to H$ is strongly measurable, since $H$ is
separable and $i$ is Borel measurable. Moreover,
\[
\int_\Omega\|i(X)\|_H\,d\mu
=\int_\Omega K(X,X)^{1/2}\,d\mu<+\infty.
\]
Hence $E_\mu(i(X))$ is well defined as a Bochner integral. Recall that $\rho_X=f_Xd\beta=X_*\mu.$ But for all $f\in H$, 

\begin{equation}
\begin{array}{r c l}
\displaystyle E_\mu(\langle i(X), f\rangle_H)=&\displaystyle E_\mu(f(X))=\int_\mathcal{X} f dX_*\mu = \int_\mathcal{X} f \cdot f_X d\beta\\
& \\
=&\displaystyle \int_\mathcal{X} P^{L^2(\mathcal{X},\beta)}_H f  \cdot f_X d\beta= \int_\mathcal{X} f \cdot P^{L^2(\mathcal{X},\beta)}_H f_X d\beta.
\end{array}
\end{equation}
 Here, the fourth equality follows from the orthogonality property of the projection \(P_H^{L^2(\mathcal X,\beta)}\), since \(f\in H\). By the uniqueness of the Bochner integral, it follows that
\[
E_\mu(i(X))=P_H^{L^2(\mathcal X,\beta)}f_X.
\]

Moreover, the sequence \((i(X_i))_{i\in\mathbb N}\) consists of i.i.d. \(H\)-valued random elements with finite first moment. The strong law of large numbers for random elements (Theorem 7.7.2, p. 204, in \cite{fpca}) therefore gives
\[
\frac1N\sum_{i=1}^N K(X_i,\cdot)
=\frac1N\sum_{i=1}^N i(X_i)
\underset{a.s.}{\longrightarrow}
E_\mu(i(X))
=P_H^{L^2(\mathcal X,\beta)}f_X.
\]
This concludes the proof.

\end{proof}

\subsection{Supervised learning}\label{sec:framework}

We focus on binary classification for the theoretical analysis. The underlying density-based principle extends naturally to multiclass classification by estimating a density score for each class and assigning each observation to the class with the largest score; see, for instance, the Parzen-window classifier \cite{jain1988classifier}. We work in the classical framework of \cite{cuck}. Let $\rho$ be a Borel probability measure on $\mathcal{X}\times\mathcal{Y}$. We suppose that $\rho$ is the law of a random variable $(X,Y):(\Omega,\mu)\rightarrow\mathcal{X}\times\mathcal{Y}$, that is,
$
\displaystyle (X,Y)_*\mu=\rho.
$
Define $$\displaystyle f_m:=\mu(Y=1)f_{X|Y=1}-\mu(Y=-1)f_{X|Y=-1},$$ where $f_{X|Y=\pm1}$ denote the conditional densities of $X$ given $Y=\pm 1$ with respect to $\beta$. Note that $f_m$ is the density,  with respect to $\beta$, of the signed measure $\displaystyle A\mapsto E_\mu(Y\mathds{1}_{\left\{X\in A\right\}})$. In particular, $|f_m|\leq f_X$, and hence $f_m\in L^2(\mathcal{X},\beta)$ since $f_X\in L^2(\mathcal{X},\beta).$
In this setting, the preceding framework can be reformulated in the following proposition:

\begin{pro}\label{lem:alternative}
Let $(X,Y):(\Omega,\mu)\to\mathcal X\times\{-1,1\}$ be a random variable. Then, possibly after extending the underlying probability space to $(\widetilde\Omega,\widetilde\mu)$, there exists a family of random variables $(Y_x)_{x\in\mathcal X}$, with $$\displaystyle Y_x:\widetilde\Omega\rightarrow \{-1,1\},\qquad x\in\mathcal X,$$ such that the map $p:x\mapsto \widetilde\mu(Y_x=1)$ is measurable on $\mathcal{X}$, and
$$
p(X)
=
\mathbb E_\mu\bigl(\mathds{1}_{\{Y=1\}}\mid X\bigr)
\qquad\text{a.s.}
$$
Moreover, the map $Y_X:\tilde{\Omega}\rightarrow\{-1,1\}$ defined by $\displaystyle Y_X(\omega):=Y_{X(\omega)}(\omega)$ is measurable and satisfies
$
\displaystyle (X,Y_X)\overset{\mathcal D}{=}(X,Y).
$

\end{pro}

\begin{proof}
Let $(\widetilde\Omega,\widetilde{\mathcal T},\widetilde\mu)
=(\Omega\times[0,1],\mathcal T\otimes\mathcal B([0,1]),
\mu\otimes\lambda)$, where $\lambda$ denotes the Lebesgue measure on
$[0,1]$. We identify $X$ with its lift to $\widetilde\Omega$, and let
$U(\omega,u)=u$. Then $U$ is uniformly distributed on $[0,1]$ and
independent of $X$.

There exists a measurable function $p:\mathcal X\to[0,1]$ such that
$
\displaystyle p(X)=E_\mu(\mathds{1}_{\{Y=1\}}\mid X)
$ a.s.
Define
\[
Y_x(\omega,u):=2\mathds{1}_{\{u\leq p(x)\}}-1,
\qquad x\in\mathcal X.
\]
Then $Y_X=2\mathds{1}_{\{U\leq p(X)\}}-1$ is measurable. Moreover, for
every Borel set $A\subset\mathcal X$,
\[
\widetilde\mu(X\in A,Y_X=1)
=\int_A \lambda([0,p(x)])\,d\rho_X(x)
=\int_A p(x)\,d\rho_X(x)
=\mu(X\in A,Y=1),
\]
where $\rho_X:=X_*\mu$. The same identity holds for $Y_X=-1$. Hence $(X,Y_X)$ and $(X,Y)$
have the same joint distribution.
\end{proof}

Using the notations of Subsection \ref{sec:random}, we can state the following results, which generalize Proposition \ref{pro:estimator1}, the latter corresponding to the case $Y=1$. 

\begin{pro}\label{pro:estimator2}
We have  \begin{equation}\label{eq:estimator}
\displaystyle E_\mu(YK(X,\cdot))=P^{L^2(\mathcal{X},\beta)}_{H}f_m.
\end{equation}
Moreover, let $\displaystyle\left((X_i, Y_i)\in(\mathcal{X}\times\mathcal{Y})^\Omega,i\in\mathbb{N}\right)$ be a sequence of i.i.d. random variables, such that $X_1\in L^2(\Omega,\mathcal{X},\mu)$ and $\displaystyle \int_\Omega K(X_1,X_1)^\frac{1}{2}d\mu<+\infty$. Then
\begin{equation}\label{eq:empiricalConv}\displaystyle \lim_{N\rightarrow + \infty}\frac{1}{N}\sum_{i=1}^N Y_i K(X_i,\cdot)\underset{a.s.}{=}P^{L^2(\mathcal{X},\beta)}_{H} f_m.\end{equation}
\end{pro}

\begin{proof}

By the same method as in the proof of Proposition \ref{pro:estimator1}, for $f\in H$, $$\displaystyle\left\langle E_\mu(YK(X,\cdot)),f\right\rangle_H
=E_\mu\left(Y\left\langle K\left(X,\cdot\right),f\right\rangle_H\right)
=E_\mu(Yf(X)) = \int_{\mathcal{X}\times \{-1,1\}}yf(x)d\rho $$ 

$$\displaystyle = \mu(Y=1)\int_\mathcal{X} f\cdot f_{X|Y=1}d\beta-\mu(Y=-1)\int_\mathcal{X} f\cdot f_{X|Y=-1}d\beta = \int_\mathcal{X}f\cdot f_md\beta=  \left\langle  P^{L^2(\mathcal{X},\beta)}_{H} f_m, f\right\rangle_H.$$

So when $N\rightarrow +\infty$, by the law of large numbers for random elements (Theorem 7.7.2 p.204 in \cite{fpca}), since $E(\|i(X)\|)= E(K(X,X)^\frac{1}{2})<+\infty$,

$$\frac{1}{N}\sum_{i=1}^N Y_i K(X_i,\cdot)
\underset{a.s.}{\longrightarrow}
E_\mu(YK(X,\cdot))
=
P^{L^2(\mathcal{X},\beta)}_H f_m.$$

\end{proof}

\begin{rmq}
Let us give an alternative way to get \eqref{eq:estimator}. For all $z\in\mathcal{X}$, we have $\displaystyle E_\mu(YK(X,z))=\int_\mathcal{X} K(x,z)f_m(x)d\beta(x).$ Theorem \ref{dirac} 2) implies that this last integral is equal to $\displaystyle P^{L^2(\mathcal{X},\beta)}_Hf_m(z)$.
\end{rmq}





\section{Bayesian machine learning via nested RKHSs}\label{sec:bayesian}

In this section, we apply the framework of nested RKHSs to Bayesian machine learning. More precisely, our goal is to extend the results of Section \ref{subsec:statistical} to the nested setting in order to recover $f_X$ or $f_m$, rather than only their projections onto a given function space, as in Propositions \ref{pro:estimator1} and \ref{pro:estimator2}. We first introduce some notation that will be used throughout this section.\\

Suppose we are given an infinite labeled sequence
$\displaystyle
S^X_\infty = (x_1, x_2, \ldots) \in \mathcal{X}^{\mathbb{N}},
$
and a random sample
$\displaystyle
\Sigma_\infty^X = (X_1, X_2, \ldots),
$
where $(X_i)_{i\in\mathbb N}$ is a sequence of i.i.d. random variables
taking values in $\mathcal X$, realized on the product probability
space introduced in Section \ref{subsec:statistical}. For each \(N \in \mathbb{N}^*\), we denote by
\[
S^X_N = (x_1, \ldots, x_N) \quad \text{and} \quad \Sigma_N^X = (X_1, \ldots, X_N)
\]
the corresponding finite subsequences. Given an infinite sequence \(\omega = (\omega_1, \omega_2, \ldots) \in \Omega^{\mathbb{N}}\), or in the finite case \(\omega = (\omega_1, \ldots, \omega_N) \in \Omega^N\), we write
\[
\Sigma_N^X(\omega) := \big( X_1(\omega_1), \ldots, X_N(\omega_N) \big).
\]
Similarly, we consider labeled sequences
\[
S_\infty := \big\{ (x_i, y_i) \in \mathcal{X} \times \mathcal{Y} : i \in \mathbb{N} \big\}
\quad \text{and} \quad
\Sigma_\infty := \big\{ (X_i, Y_i) \in (\mathcal{X} \times \mathcal{Y})^\Omega : i \in \mathbb{N} \big\},
\]
where \(\Sigma_\infty\) is a sequence of i.i.d.\ random variables. If \(\pi_N : (\mathcal{X} \times \mathcal{Y})^\mathbb{N} \to (\mathcal{X} \times \mathcal{Y})^N\) denotes the projection onto the first \(N\) pairs, we define
\[
S_N := \pi_N S_\infty \quad \text{and} \quad \Sigma_N := \pi_N \Sigma_\infty.
\]

Let $(H,(H_n)_n,(K_n)_n)$ be a nested RKHS. Throughout this section, we assume that $\displaystyle\int_\Omega K_n(X_1,X_1)^{\frac12}d\mu<+\infty$ for every $n\in\mathbb{N}$, so that Propositions \ref{pro:estimator1} and \ref{pro:estimator2} apply to each $H_n$ (this holds, for instance, when each $K_n$ is bounded on the diagonal). We place ourselves in the case of binary classification. 
We use the same notations of Subsection \ref{sec:framework}. Expressed in this context, Propositions \ref{pro:estimator1} and \ref{pro:estimator2} allow us to formulate:

\begin{pro}\label{pro:def_fn}
Fix $n\in\mathbb N$. For each $N\in\mathbb N^*$, define
\[
f^*_{n,\Sigma^X_N}:=
\frac1N\sum_{i=1}^N K_n(X_i,\cdot),
\qquad
f^*_{n,\Sigma_N}:=
\frac1N\sum_{i=1}^N Y_iK_n(X_i,\cdot).
\]
Then $f^*_{n,\Sigma^X_N}$ and $f^*_{n,\Sigma_N}$ are unbiased
estimators of $P^{L^2(\mathcal X,\beta)}_{H_n}f_X$ and
$P^{L^2(\mathcal X,\beta)}_{H_n}f_m$, respectively. Moreover, as $N\rightarrow +\infty,$
\[
f^*_{n,\Sigma^X_N}\underset{a.s.}{\longrightarrow}
P^{L^2(\mathcal X,\beta)}_{H_n}f_X
\text{  and  }
f^*_{n,\Sigma_N}\underset{a.s.}{\longrightarrow}
P^{L^2(\mathcal X,\beta)}_{H_n}f_m.
\]
\end{pro}

In principle, we can recover $f_m$ by first letting
$N\rightarrow+\infty$ and then $n\rightarrow+\infty$:
for fixed $n$, the empirical estimator converges to
$P^{L^2(\mathcal X,\beta)}_{H_n}f_m$, while
\[
P^{L^2(\mathcal X,\beta)}_{H_n}f_m
\longrightarrow f_m
\quad\text{in }L^2(\mathcal X,\beta).
\]
In practice, however, this is impossible, since we do not have access
to the whole sample $\Sigma_\infty$, but only to a finite sample
$\Sigma_N$. So we will formulate the problem of finding the best approximation of $f_m$ in $H_n$ in the $L^2$-sense in terms of an empirical minimization problem. For each fixed $n$, this will provide an estimator converging to the projection of $f_m$ onto $H_n$ as $N\rightarrow+\infty$. We will then use a diagonal selection to obtain a sequence converging to $f_m$ when both $n$ and $N$ tend to infinity.\

In Subsection \ref{sub:risk}, we construct a risk functional for estimating $P^{L^2(\mathcal{X},\beta)}_{H_n}f_m$ in $H_n$, and we provide an empirical estimator. We prove the consistency of this procedure as the sample size tends to infinity. We give a geometric interpretation of this setting, which is intrinsic since it is independent of $n$, in Subsection \ref{lift}. In Subsection \ref{sub:minL2}, we prove the consistency of the estimation of $f_m$ by a sequence obtained through a diagonal selection of the empirical estimators for each $n$ and $N$. We establish sufficient optimality properties for this extracted sequence, which allow us to construct it. Finally, in Subsection \ref{sub:probabilistic}, we give a probabilistic interpretation of these constructions.

\subsection{Risk functionals}\label{sub:risk}

In this subsection, we formulate the estimation of $f_X$ and $f_m$ as the minimization of suitable risk functionals and study the consistency of the corresponding empirical risk minimization procedures.\\

Let the notation of Subsection \ref{sec:framework} prevail. Recall that $S_N=\{(x_i, y_i)\}_{i=1}^N$, $\Sigma_N=\{(X_i, Y_i)\}_{i=1}^N$ where the latter is a sequence of i.i.d. random variables with distribution $\rho$. Let $H\subset L^2(\mathcal{X},\beta)$ be an RKHS with a reproducing kernel with respect to the canonical scalar product. Let $h:\mathbb{R}^2\rightarrow \mathbb{R}$ be a measurable loss function such that $h(f(X_1),Y_1)$ is $\mu$-integrable for every $f\in H$; we define a risk functional $\mathcal{E}_{h,\rho}: H\rightarrow \mathbb{R}$ by \begin{equation}\label{eq:riskfunctional}\displaystyle \mathcal{E}_{h,\rho}(f):= \int_\Omega h(f(X_1),Y_1)d\mu = \int_{\mathcal{X}\times\mathcal{Y}}  h(f(x),y)d\rho.\end{equation}

We also define the associated empirical risk functional $\displaystyle\mathcal{E}^*_{h}:(\mathcal{X}\times \mathcal{Y})^N\times H\rightarrow \mathbb{R}$ by \begin{equation}\label{eq:empiricalrisk}\displaystyle\mathcal{E}_{h}^*(S_N,f):=\frac{1}{N}\sum_{(x_i,y_i)\in S_N} h\left(f(x_i), y_i\right).\end{equation}

For an RKHS $H\subset L^2(\mathcal{X},\beta)$ with reproducing kernel
$K$, we denote by $H^*$ the map from $\mathcal{X}^N$ to the set of
subspaces of $L^2(\mathcal{X},\beta)$, defined by
$$
\displaystyle H^*(x_1,\ldots,x_N)
=
\operatorname{span}\left(K(x_i,\cdot):1\leq i\leq N\right).
$$
When $S_N$ is fixed, we will denote $$\displaystyle\mathcal{E}_{h,{S_N}}^*:f\mapsto \mathcal{E}_{h}^*(S_N,f).$$ The following proposition shows that empirical risk minimization over $H$ reduces to minimization over the finite-dimensional subspace $H^*(S_N)$.

\begin{pro}[\cite{cuck2} Corollary 2.26]\label{lemempiric}Let $D\subset H$ be a closed subset. Let $n\in \mathbb{N}$, and $\mathcal{E}_{h,S_N}^*:H\rightarrow\mathbb{R}$ an empirical risk functional defined by \eqref{eq:empiricalrisk}. We have
\begin{equation}\label{eq:sufficient}
\displaystyle \inf_{f\in D}\mathcal{E}_{h,S_N}^*(f)=\inf_{f\in P^H_{H^*(S_N)}(D)}\mathcal{E}_{h,S_N}^*(f).
\end{equation}
\end{pro}

\begin{proof} Let $f\in D$. For all $(x,y)\in S_N$,  $f(x) = \langle f \mid K(x,\cdot) \rangle =\langle  P^{H}_{H^*(S_N)}f \mid K(x,\cdot) \rangle = (P^{H}_{H^*(S_N)}f)(x)$. Thus $\mathcal{E}^*_{h,S_N}(f)=\mathcal{E}^*_{h,S_N}(P^{H}_{H^*(S_N)}f)$. Equality \eqref{eq:sufficient} follows using a minimizing sequence.\end{proof}

\noindent We now specialize Proposition \ref{lemempiric} to the loss $h(x,y)=h_1(x,y):=1-xy.$ From now on, we write $\displaystyle\mathcal{E}_{1,\rho}:=\mathcal{E}_{h_1,\rho}$ and $\displaystyle\mathcal{E}^*_{1,\Sigma_N}:=\mathcal{E}^*_{h_1,\Sigma_N}$ to simplify the notation. The following identity will play a key role.


\begin{lem}
Let $f\in L^2(\mathcal{X},\beta).$ Then \begin{equation}\label{eq:reformulation}\displaystyle \mathcal{E}_{1,\rho}(f)=1-\langle f , f_m\rangle_{L^2(\mathcal{X},\beta)}.\end{equation}
\end{lem}


\begin{proof}
Since $\rho_X\ll\beta$, $f(X_1)$ is well defined $\mu$-a.s. for $f\in L^2(\mathcal{X},\beta)$, and $E_\mu|f(X_1)|=\int_{\mathcal{X}}|f|f_X\,d\beta<+\infty$. Then
\begin{align*}
 \displaystyle1-\mathcal{E}_{1,\rho}(f)&\displaystyle= E_\mu(Y_1f(X_1))=\int_{\mathcal{X}\times \mathcal{Y}}yf(x)\,d\rho(x,y)\\
 &\displaystyle=\sum_{j=\pm 1}j\,\mu(Y=j)\int_\mathcal{X}f(x)\,f_{X\mid Y=j}(x)\,d\beta(x)\\
 &\displaystyle=\int_\mathcal{X}f(x) f_m(x)\,d\beta(x)=\langle f , f_m\rangle_{L^2(\mathcal{X},\beta)}.
\end{align*}
\end{proof}

The previous representation reduces the minimization of $\mathcal{E}_{1,\rho}$ to the maximization of a linear functional over $H$. The following theorem characterizes the maximizers of such functionals and will be our main tool in what follows.

\begin{thm}\label{evaluation} Let $(\alpha_1,\ldots,\alpha_N)\in \mathbb{R}^N$ and $(x_1,\ldots,x_N)\in\mathcal{X}^N$. Define $\displaystyle g:=\sum_{i=1}^N\alpha_i K(x_i,\cdot),$ and suppose $g\neq 0$. Then $$\displaystyle \left\|g\right\|_{L^2(\mathcal{X},\beta)}=\underset{f\in H\setminus\{0\}}{\sup}\frac{\left|\sum_{i=1}^N \alpha_i f(x_i)\right|}{\left\|f\right\|_{L^2(\mathcal{X},\beta)}}=\sqrt{\sum_{i,j=1}^N \alpha_i\alpha_j K(x_i,x_j)}=\underset{f\in H^*(x_1,\ldots,x_N)\setminus\{0\}}{\sup}\frac{\left|\sum_{i=1}^N\alpha_i f(x_i)\right|}{\left\|f\right\|_{L^2(\mathcal{X},\beta)}}.$$
Moreover, the supremum is attained precisely by the nonzero scalar multiples of $g$.

\noindent Consequently, the linear functional $\displaystyle f\mapsto \sum_{i=1}^N \alpha_i\delta _{x_i}(f)$ on $H$ has norm $$\displaystyle \sqrt{\sum_{i,j=1}^N\alpha_i\alpha_j K(x_i,x_j)}=\left\|\sum_{i=1}^N\alpha_iK(x_i,\cdot)\right\|_{L^2(\mathcal{X},\beta)}\geq 0.$$
\end{thm}

\begin{proof}
All assertions follow immediately from the Cauchy--Schwarz inequality, which yields, for every $f\in H$, 

\begin{align*}
\left|\sum_{i=1}^N\alpha_i f(x_i)\right|
=
\left|
\left\langle
\sum_{i=1}^N\alpha_iK(x_i,\cdot),f
\right\rangle_{L^2(\mathcal X,\beta)}
\right|   &\le
\left\|
\sum_{i=1}^N\alpha_iK(x_i,\cdot)
\right\|_{L^2(\mathcal X,\beta)}
\,\left\|f\right\|_{L^2(\mathcal X,\beta)}      \\
&=
\sqrt{\sum_{i,j=1}^N
\alpha_i\alpha_jK(x_i,x_j)}
\,\left\|f\right\|_{L^2(\mathcal X,\beta)}.
\end{align*}
Equality holds if and only if the two vectors are collinear.
\end{proof}

\begin{rmq}The function \(g\) in Theorem \ref{evaluation} is the unique minimum-norm interpolant among all functions in \(H\) taking the same values as \(g\) at \(x_1,\ldots,x_N\). Indeed, for any such interpolant \(f\), we have \(f-g\perp \operatorname{span}\{K(x_i,\cdot):i=1,\ldots,N\}\). Since \(g\) belongs to this span, the Pythagorean identity yields
$
\displaystyle \left\|f\right\|_H^2=\left\|g\right\|_H^2+\left\|f-g\right\|_H^2,
$
which proves the claim.
\end{rmq}

\begin{rmq}\label{rmq:geqone}
Thanks to the construction of Theorem \ref{th:distinguished}, there are many RKHS $H\subset L^2(\mathcal{X},\beta)$ with reproducing kernels
$K$  with respect to the canonical scalar product such that $1\in H$.
In this case, we can derive some interesting inequalities.
Taking $\alpha_i=1$ in Theorem \ref{evaluation}, we obtain
\begin{equation}\label{eq:nonvanishing}
N\leq
\sup_{f\in H,\ \left\|f\right\|_H=1}
\left|\sum_{i=1}^N f(x_i)\right|
=
\sqrt{\sum_{i,j=1}^N K(x_i,x_j)}.
\end{equation}
Moreover, taking $\alpha_i=e_i$ and using $1\in H$ yields
$K(x_i,x_i)\geq 1$.
\end{rmq}

Now we apply the previous considerations to the framework presented in Subsection \ref{subsec:statistical}. Let $\displaystyle B:=\left\{g\in H, \left\|g\right\|_{L^2(\mathcal{X},\beta)}= 1\right\}$ be the unit sphere. Since $\Sigma_N$ is random, empirical minimizers can naturally be regarded as $B$-valued random variables, that is, as elements of $\mathcal{M}(\Omega,B)$. Note that \eqref{eq:reformulation}, together with the fact that $\displaystyle\left\|f_m\right\|_{L^2(\mathcal{X},\beta)}\leq 1$ by construction, implies that $\displaystyle\mathcal{E}_{1,\rho}:B\rightarrow\mathbb R$ takes values in $\mathbb{R}^+$.

\begin{cor}\label{cor:riskdensity}
Let $N\in\mathbb{N}$. For every sample $S_N=((x_i,y_i))_{i=1}^N$, define $$\displaystyle f^*_{S_N}:=\frac{1}{N}\sum_{i=1}^N y_i K(x_i,\cdot).$$ Then, provided that the denominators below do not vanish,
\begin{equation}\label{eq:empiricalepsilonone}\displaystyle \frac{P^{L^2(\mathcal{X},\beta)}_{H}f_m}{\left\|P^{L^2(\mathcal{X},\beta)}_{H}f_m\right\|_{L^2(\mathcal{X},\beta)}} \in\arg\min_{f\in B} \mathcal{E}_{1,\rho}\quad\text{}\end{equation}
and, for every $\omega\in\Omega^N$, 

\begin{equation}\label{eq:globalmin}
\left\{\frac{ f^*_{\Sigma_N(\omega)} }{\left\|f^*_{\Sigma_N(\omega)}\right\|_{L^2(\mathcal{X},\beta)}}\right\} =\arg\min_{f\in B}\mathcal{E}^*_{1,\Sigma_N(\omega)}.\end{equation}
\end{cor}

\begin{proof}
Let $f\in B$. Fix $\omega=(\omega_1,\ldots,\omega_N)\in \Omega^N$. Applying Theorem \ref{evaluation} to the sample $\Sigma_N(\omega)$ with $\alpha_i=Y_i(\omega_i)$ and $x_i=X_i(\omega_i)$ yields the inequality $$\displaystyle \sum_{i=1}^N Y_i(\omega_i) f(X_i(\omega_i))\leq \sum_{i=1}^N Y_i(\omega_i) \frac{f^*_{\Sigma_N(\omega)}(X_i(\omega_i))}{\left\|f^*_{\Sigma_N(\omega)}\right\|_{L^2(\mathcal{X},\beta)}},$$ with equality if and only if $\displaystyle f=\frac{f^*_{\Sigma_N(\omega)}}{\left\|f^*_{\Sigma_N(\omega)}\right\|_{L^2(\mathcal{X},\beta)}}$. In other words, $$\displaystyle \mathcal{E}^*_{1,\Sigma_N(\omega)}(f)\geq \mathcal{E}^*_{1,\Sigma_N(\omega)}\left(\frac{f^*_{\Sigma_N(\omega)}}{\left\|f^*_{\Sigma_N(\omega)}\right\|_{L^2(\mathcal{X},\beta)}}\right),$$ which proves \eqref{eq:globalmin}. Moreover,  \eqref{eq:reformulation} gives 

\begin{align*}
\displaystyle \mathcal{E}_{1,\rho}(f)=1-\left\langle f,f_m\right\rangle_{L^2(\mathcal{X},\beta)}\geq & 1-\left\langle \frac{P^{L^2(\mathcal{X},\beta)}_H f_m}{\left\|P^{L^2(\mathcal{X},\beta)}_H f_m\right\|_{L^2(\mathcal{X},\beta)}},f_m\right\rangle_{L^2(\mathcal{X},\beta)}\\
\displaystyle =&\mathcal{E}_{1,\rho}\left( \frac{P^{L^2(\mathcal{X},\beta)}_H f_m}{\left\|P^{L^2(\mathcal{X},\beta)}_H f_m\right\|_{L^2(\mathcal{X},\beta)}}\right).
\end{align*}

\end{proof}

\noindent\textbf{Study of consistency}. By slight abuse of notation, for $f\in\mathcal{M}(\Omega, B)$, we still write $\mathcal{E}_{h,\Sigma_N}^*(f)$ for the random variables $\omega\mapsto \mathcal{E}_{h,\Sigma_N(\omega)}^*(f(\omega))$, and we write $f^*_{\Sigma_N}$ for the random element $\displaystyle\omega\mapsto f^*_{\Sigma_N(\omega)}$. Following Vapnik \cite{vapnik1998statistical} Chapter 3, given a risk functional $\mathcal{E}_{h,\rho}:B\rightarrow\mathbb{R}$, provided that for each $N\in \mathbb{N}$ and each $S_N\in(\mathcal{X}\times\mathcal{Y})^N$, the set of minimizers $\displaystyle \displaystyle\underset{f\in B}{\text{argmin }}\mathcal{E}_{h,S_N}^*(f)$ is nonempty, the empirical risk minimization procedure is said to be \textit{consistent} if the following convergence holds in probability: $$\displaystyle \lim_{N\rightarrow +\infty}\mathcal{E}_{h,\rho}(f_N)= \underset{f\in B}{\inf}\mathcal{E}_{h,\rho}(f)$$ 
for any sequence of empirical risk minimizers $$ f_N\in\displaystyle\underset{f\in \mathcal{M}(\Omega,B)}{\text{argmin }}\mathcal{E}_{h,{\Sigma}_N}^*(f).$$ A sufficient condition is that $(f_N)_N$ converges almost surely to a minimizer of $\mathcal{E}_h$ in $H$, in which case the procedure will be said to be \textit{strongly consistent}.

\begin{cor}
Assume that \(f_m\neq 0\). The minimization problem of \(\mathcal E_{1,\rho}\) is strongly consistent, provided that \(f^*_{n_N,\Sigma_N}\neq 0\) almost surely for every \(N\in\mathbb N\). Moreover,
\[
\frac{f^*_{n_N,\Sigma_N}}{\|f_m\|_{L^2(\mathcal X,\beta)}}
\xrightarrow[N\to\infty]{\mathrm{a.s.}}
\frac{f_m}{\|f_m\|_{L^2(\mathcal X,\beta)}},
\]
and
\[
\mathcal E_{1,\rho}\left(
\frac{f^*_{n_N,\Sigma_N}}{\|f_m\|_{L^2(\mathcal X,\beta)}}
\right)
\xrightarrow[N\to\infty]{\mathrm{a.s.}}
\mathcal E_{1,\rho}\left(
\frac{f_m}{\|f_m\|_{L^2(\mathcal X,\beta)}}
\right).
\]

\end{cor}

\begin{proof}
The first assertion follows from Theorem \ref{thm:unifcrmconv}, the continuity of the normalization map, and \eqref{eq:globalmin}. The remaining convergences follow from Theorem \ref{thm:unifcrmconv} and the continuity of \(\mathcal E_{1,\rho}\) on \(L^2(\mathcal X,\beta)\), established by \eqref{eq:reformulation}.
\end{proof}

\subsection{Geometric interpretation}\label{lift}

In this subsection, we show how the canonical embedding of $\mathcal X$ into $H_n$ allows a decision boundary in $\mathcal X$ to be interpreted as a hyperplane in $H_n$. This viewpoint provides a geometric interpretation of the minimization of the cost function $\mathcal E_{1,S_N}$ over $B$ as an optimization problem involving distances to this hyperplane.\\
For every $n$, the reproducing kernel defines the canonical feature map
\[
i_n:\mathcal{X}\longrightarrow H_n,\qquad
i_n(x)=K_n(x,\cdot).
\]
For $h\in H_n$, let
\[
h^\perp:=\{u\in L^2(\mathcal X,\beta):\langle h,u\rangle_{L^2(\mathcal{X},\beta)}=0\},
\]
and define
\[
h^{\perp,n}:=h^\perp\cap H_n.
\]
By the reproducing property, $\displaystyle h(x)=\langle h, i_n(x)\rangle_{L^2(\mathcal{X},\beta)},$ so
\[
\{x\in \mathcal{X}:h(x)=0\}
=
i_n^{-1}\!\left(h^\perp\cap i_n(\mathcal{X})\right).
\]
Hence, the decision boundary of $h$ is the trace of the hyperplane $h^\perp$ on $i_n(\mathcal{X})$. \\
For \(0\neq h\in H_{n_0}\), the following proposition shows that the distance from \(i_n(x)\) to the hyperplane \(h^{\perp,n}\) in \(H_n\) depends only on \(h\), and not on the ambient RKHS \(H_n\), for every \(n\geq n_0\).

\begin{pro}\label{pro:distance}
Let $0\neq h\in H_{n_0}$ for some $n_0\in \mathbb{N}$.
For every $n\geq n_0$, the distance in $H_n$ from the point $i_n(x)$ to the hyperplane
$h^{\perp,n}$ is
\begin{equation}\label{eq:distanceHn}
d_{H_n}\left(i_n(x),h^{\perp,n}\right)
=
\frac{|h(x)|}{\left\|h\right\|_{L^2(\mathcal{X},\beta)}}.
\end{equation}
In particular, this distance does not depend on the ambient RKHS $H_n$. Moreover, the distance in $L^2(\mathcal{X},\beta)$ from the point $i_n(x)$ to the hyperplane
$h^{\perp}\subset L^2(\mathcal{X},\beta)$ is \begin{equation}\label{eq:distanceH}\displaystyle d_{L^2(\mathcal{X},\beta)}\left(i_n(x),h^\perp\right)= d_{H_n}\left(i_n(x),h^{\perp}\cap H_n\right).\end{equation}

\end{pro}

\begin{proof} 
The unit normal vector to $h^{\perp,n}$ in $H_n$ is $\displaystyle u=\frac{h}{\left\|h\right\|_{L^2(\mathcal{X},\beta)}}.$ Hence
\begin{equation}\label{eq:geometric}
d_{H_n}(i_n(x),h^{\perp,n})
=
\left|
\left\langle
i_n(x),u
\right\rangle_{H_n}
\right|
=
\frac{
|\langle i_n(x),h\rangle_{H_n}|
}{
\left\|h\right\|_{L^2(\mathcal{X},\beta)}
}.
\end{equation}
The reproducing property yields $\displaystyle \langle i_n(x),h\rangle_{H_n}=h(x)$ which proves \eqref{eq:distanceHn}. Since $\displaystyle L^2(\mathcal{X},\beta)=h^\perp\oplus \mathrm{span}(h)$, and $\displaystyle \mathrm{span}(h)\subset H_n$, the orthogonal projection of every point of $H_n$ onto $h^\perp$ still belongs to $H_n$. Therefore \eqref{eq:distanceH} follows.
\end{proof}


For binary classification, the sign of $h(x)$ indicates on which side of the hyperplane $h^\perp$ the point $i_n(x)$ lies. Therefore, $\mathcal{E}_{1,S_N}^*(h)$, for $h\in H_{n_0}\cap B$, admits the following geometric interpretation: by \eqref{eq:geometric}, $$1-\displaystyle\mathcal{E}_{1,S_N}^*(h)=\frac{1}{N}\sum_{(x_i,y_i)\in S_N}y_i\cdot \mathrm{sign}(h(x_i))\cdot d_{L^2(\mathcal{X},\beta)}(i_n(x_i), h^\perp).$$
By Proposition \ref{pro:distance}, this quantity is independent of $n$ for every $n\geq n_0$. The factor $y_i\cdot \mathrm{sign}(h(x_i))$ takes values in $\{-1,+1\}$: it equals $+1$ if $x_i$ is correctly classified by $h$, and $-1$ otherwise. Consequently, minimizing $\mathcal{E}_{1,S_N}^*(h)$ on $B$ is equivalent to maximizing the total distance from the training points in $S_N$ to the hyperplane $h^\perp$, with positive contributions from correctly classified points and negative contributions from misclassified ones. Hence the optimal hyperplane ${f_{n,S_N}^*}^{\perp}$ is precisely the one maximizing this total signed distance.







\subsection{Minimization of $\mathcal{E}_{1,\rho}$ in $L^2(\mathcal{X},\beta)$.}\label{sub:minL2}

In this subsection, we prove the existence of a function $(N,\omega)\mapsto n_N(\omega)$ from $\mathbb N\times\Omega^{\mathbb{N}}$ to $\mathbb N$ such that $f^*_{n_N,\Sigma_N}\to f_m$ a.s. as $N\to+\infty$.
This is done in Subsubsection \ref{sec:premiminary}, where we prove the existence of a minimizing sequence for the risk functional $\mathcal E_{1,\rho}$ in $L^2(\mathcal X,\beta)$. The consistency of this minimization procedure is then studied. In Subsubsection \ref{sec:consistence}, we provide criteria for estimating the terms of this sequence in a machine-learning context.\\

\noindent Let $(\Omega, \mathcal{T}, \mu)$ be a probability space, and $(H, H_n,K_n)$ a nested RKHS.

\subsubsection{Existence of a minimizing sequence}\label{sec:premiminary}

Fix $N\in\mathbb N$. A deterministic sample $S_N\in (\mathcal{X}\times\mathcal{Y})^N$ being given, define
\[
f^*_{n,S_N}:=
\frac1N\sum_{i=1}^N y_i K_n(x_i,\cdot).
\]

\begin{lem}\label{lem:oneminimum}
The function $\displaystyle n\mapsto \left\|f^*_{n,S_N}\right\|_{L^2(\mathcal{X},\beta)}$ is non-decreasing. 

\end{lem}

\begin{proof}

It suffices to note that $f^*_{n-1,{S_N}}\perp f^*_{n,{S_N}}-f^*_{n-1,{S_N}}$ in $H$, that is consequence of \eqref{eq:perp}. Thus  \begin{equation}\label{eq:firstpartweak}
\left\|f^*_{n,{S_N}}\right\|_{L^2(\mathcal{X},\beta)}^2-\left\|f^*_{n-1,{S_N}}\right\|_{L^2(\mathcal{X},\beta)}^2=\left\|f^*_{n,{S_N}}-f^*_{n-1,{S_N}}\right\|_{L^2(\mathcal{X},\beta)}^2\geq 0.
\end{equation}

\end{proof}

\begin{rmq}
The statement of Lemma \ref{lem:oneminimum} can also be deduced from Theorem \ref{evaluation}. Indeed,

$$
\displaystyle
\left\|f^*_{n,S_N}\right\|_{L^2(\mathcal{X},\beta)}
=
\frac{1}{N}
\sup_{f\in H_n,\|f\|_\beta=1}
\left|\sum_{i=1}^N y_i f(x_i)\right|.
$$

Since $(H_n)_n$ is nested, the right-hand side is non-decreasing in $n$.
\end{rmq}

Let $S\in ({\mathcal{X}\times \mathcal{Y}})^\mathbb{N}$. Define $\displaystyle \mathcal{E}_0:H\rightarrow \mathbb{R}$ by $\displaystyle \mathcal{E}_0(f):=\|f- f_m\|$. For $\epsilon>0$, we can define $$\displaystyle n_{\epsilon,S_N}:=\min\left\{{n\in\mathbb{N}}:\mathcal{E}_0(f^*_{n,S_N})<\inf_{k\in\mathbb{N}}\mathcal{E}_0(f^*_{k,S_N})+\epsilon\right\}.$$
This quantity is well defined since the set in the definition is nonempty by construction. This choice is measurable with respect to $S_N$, since the infimum is taken over a countable set and the minimum is taken over $\mathbb N$. We denote by $n_{\epsilon,\Sigma_N}$ the corresponding random variable $
n_{\epsilon,\Sigma_N}(\omega):=n_{\epsilon,\Sigma_N(\omega)}.$

\begin{thm}\label{thm:unifcrmconv}
Assume that $H=L^2(\mathcal X,\beta)$, and set

$$
n_N:=n_{1/N,\Sigma_N}.
$$
Then, in the sense of Definition \ref{def:almostsure},
\begin{equation}\label{eq:unifcrmconv}
\lim_{N\rightarrow+\infty}f^*_{n_N,\Sigma_N}
\underset{a.s.}{=}
f_m.
\end{equation}

\end{thm}

\begin{proof}

We want to prove that there exists $E\subset\Omega^{\mathbb N}$ with
$\mu^\infty(E)=1$ such that, for every $\omega\in E$, $$\lim_{N\to+\infty}f^*_{n_N(\omega),\Sigma_N(\omega)}\underset{L^2(\mathcal X,\beta)}{=}f_m.$$
Fix $\epsilon>0$. Since $\bigcup_n H_n$ is dense in $H=L^2(\mathcal X,\beta)$, there exists $n_1\in\mathbb N$ such that

$$
\left\|f_m-P^{L^2(\mathcal X,\beta)}_{H_{n_1}}f_m\right\|_{L^2(\mathcal{X},\beta)}<\epsilon.
$$

By Proposition \ref{pro:estimator2}, there exists a set
$E\subset\Omega^{\mathbb N}$ with $\mu^\infty(E)=1$ such that, for every
$\omega\in E$,

$$
f^*_{n_1,\Sigma_N(\omega)}
\underset{L^2(\mathcal X,\beta)}{\longrightarrow}
P^{L^2(\mathcal X,\beta)}_{H_{n_1}}f_m.
$$

Hence, for every $\omega\in E$,

$$
\mathcal E_0\left(f^*_{n_1,\Sigma_N(\omega)}\right)
\longrightarrow
\left\|P^{L^2(\mathcal X,\beta)}_{H_{n_1}}f_m-f_m\right\|_{L^2(\mathcal{X},\beta)}
<\epsilon.
$$

By the definition of $n_N=n_{1/N,\Sigma_N}$,

$$
\mathcal E_0\left(f^*_{n_N,\Sigma_N}\right)
\leq
\inf_{n\in\mathbb N}\mathcal E_0(f^*_{n,\Sigma_N})+\frac1N
\leq
\mathcal E_0\left(f^*_{n_1,\Sigma_N}\right)+\frac1N.
$$

Therefore, for every $\omega\in E$,

$$
\limsup_{N\to+\infty}
\mathcal E_0\left(f^*_{n_N(\omega),\Sigma_N(\omega)}\right)
\leq\epsilon.
$$

The set $E=E_\epsilon$ depends on $\epsilon$ (through $n_1$). Applying the above with $\epsilon=1/k$, $k\in\mathbb{N}^*$, and replacing $E$ by $\bigcap_{k}E_{1/k}$, which still has full measure, we obtain, for every $\omega$ in this set,

$$
\lim_{N\to+\infty}
\mathcal E_0\left(f^*_{n_N(\omega),\Sigma_N(\omega)}\right)=0.
$$

Thus $\displaystyle f^*_{n_N,\Sigma_N}
\underset{L^2(\mathcal X,\beta)}{\longrightarrow}
f_m
$ a.s., which proves the result. 
\end{proof}

We also define a less restrictive notion of consistency for the case where there is no guarantee that the minimum of the empirical risk functional is reached. The risk functional $\mathcal{E}_h$ will be said {\it{weakly consistent}} if there exists a sequence $(f_N)\in B^\mathbb{N}$ such that $$\displaystyle\left|\mathcal{E}_{h,\rho}(f_N)-\underset{f\in B}{\inf}\mathcal{E}_{h,\rho}(f)\right|\rightarrow 0$$ for any sequence $f_N$ such that $\mathcal{E}_{h,\Sigma_N}^*(f_N)\rightarrow \underset{f\in \mathcal{M}(\Omega,B)}{\inf}\mathcal{E}_{h,\Sigma_N}^*(f)$ when $N\rightarrow +\infty$. For that, it suffices that $(f_N)_N$ converges almost surely towards a minimizer of $\mathcal{E}_{h,\rho}$ in $H$, in which case we will say that the risk functional $\mathcal{E}_{h,\rho}$ is \textit{convergent}.

\begin{cor}
The minimization problem of $\mathcal{E}_{1,\rho}$ is strongly consistent, provided that $f^*_{n_N,\Sigma_N}(\omega)\neq 0$ for all $N\in \mathbb{N}$ and almost $\omega\in \Omega^\infty$. In the general case, $$\displaystyle \frac{f^*_{n_N,\Sigma_N(\omega)}}{\left\|f_m\right\|_{L^2(\mathcal{X},\beta)}}\underset{N\rightarrow +\infty}{\rightarrow} \frac{f_m}{\left\|f_m\right\|_{L^2(\mathcal{X},\beta)}}$$
and
$$\displaystyle\mathcal{E}_{1,\rho}\left(\frac{f^*_{n_N,\Sigma_N(\omega)}}{\left\|f_m\right\|_{L^2(\mathcal{X},\beta)}}\right)\underset{N\rightarrow +\infty}\rightarrow \mathcal{E}_{1,\rho}\left(\frac{f_m}{\left\|f_m\right\|_{L^2(\mathcal{X},\beta)}}\right).$$
\end{cor}

\begin{proof}
By Theorem \ref{thm:unifcrmconv}, we have the almost-sure convergence
\[
f^*_{n_N,\Sigma_N}\xrightarrow[N\to\infty]{\mathrm{a.s.}} f_m
\quad\text{in }L^2(\mathcal X,\beta).
\]
If \(f^*_{n_N,\Sigma_N}\neq 0\) almost surely for every \(N\in\mathbb N\), continuity of the normalization map at \(f_m\neq 0\) yields the almost-sure convergence of the normalized sequence to \(f_m/\|f_m\|_{L^2(\mathcal X,\beta)}\). By \eqref{eq:globalmin}, this limit is the unique minimizer of \(\mathcal E_{1,\rho}\) on the unit sphere of \(L^2(\mathcal X,\beta)\), proving strong consistency. In the general case, the stated convergence of the rescaled functions follows directly from Theorem \ref{thm:unifcrmconv}. Finally, \eqref{eq:reformulation} implies that \(\mathcal E_{1,\rho}\) is continuous on \(L^2(\mathcal X,\beta)\), which yields the convergence of the corresponding risks.
\end{proof}

\subsubsection{Estimation of the minimizing sequence}\label{sec:consistence}

In the previous subsection, we have established the existence of a function $(n,\omega)\mapsto n_{N}(\omega)$ from $\mathbb{N}\times \Omega$ to $\mathbb{N},$ which allows us to construct a sequence $\displaystyle \left(f_{n_N,\Sigma_N}\right)_N$ of good approximates of $f_m$. In this subsection, we give two theorems providing numerical criteria for estimating this function.

\begin{thm}
 For almost every $\omega\in\Omega^\mathbb N$ and every $\epsilon>0$, there exists $N_0\in\mathbb N$ such that, for every $N\geq N_0$, there exists a function $h_{N,\omega}:\mathbb N\to\mathbb R_+$, non-increasing on $\{0,\ldots,n_N(\omega)\}$ and non-decreasing on $\{n_N(\omega),\ldots\}$, with $h_{N,\omega}(n_N(\omega))=0$, such that, for every $n\in\mathbb N$, \[ \left| \mathcal{E}_0\left(f^*_{n,\Sigma_N(\omega)}\right)-h_{N,\omega}(n) \right| \leq\epsilon\]

\end{thm}

\begin{proof}

Fix $\epsilon>0$. By Theorem \ref{thm:unifcrmconv}, for almost every $\omega\in\Omega^\mathbb{N}$, there exists $N_0=N_0(\omega,\epsilon)\in\mathbb{N}$ such that$\displaystyle\left\|f^*_{n_N(\omega),\Sigma_N(\omega)}-f_m\right\|_{L^2(\mathcal{X},\beta)}<\epsilon$ for all $N\geq N_0$. Fix such an $\omega\in\Omega^\mathbb{N}$. We have the triangular inequality:

$$\displaystyle \left| \left\| f^*_{n,\Sigma_N(\omega)}- f_m \right\|_{L^2(\mathcal{X},\beta)}- \left\| f^*_{n,\Sigma_N(\omega)}- f^*_{n_N(\omega),\Sigma_N(\omega)}\right\|_{L^2(\mathcal{X},\beta)}\right| \leq \left\|f^*_{n_N(\omega),\Sigma_N(\omega)}-f_m\right\|_{L^2(\mathcal{X},\beta)}.$$

Hence, for all $N\geq N_0$ and all $n\in\mathbb{N},$

$$\displaystyle \left\| f^*_{n,\Sigma_N(\omega)}- f^*_{n_N(\omega),\Sigma_N(\omega)}\right\|_{L^2(\mathcal{X},\beta)}-\epsilon\leq \left\| f^*_{n,\Sigma_N(\omega)}- f_m \right\|_{L^2(\mathcal{X},\beta)}\leq \left\| f^*_{n,\Sigma_N(\omega)}- f^*_{n_N(\omega),\Sigma_N(\omega)}\right\|_{L^2(\mathcal{X},\beta)}+\epsilon.$$

On the other hand, denoting $$\displaystyle h_{N,\omega}:n\mapsto \left\| f^*_{n,\Sigma_N(\omega)}- f^*_{n_N(\omega),\Sigma_N(\omega)}\right\|_{L^2(\mathcal{X},\beta)},$$ we have $h_{N,\omega}(n)=0$ for $n=n_N(\omega)$, and else we compute:

\begin{equation}\displaystyle h_{N,\omega}^2(n)=\frac{1}{N^2}\left\|\sum_{i=1}^N\left(Y_i(\omega_i) K_n(X_i(\omega_i),\cdot)-Y_i(\omega_i) K_{n_N(\omega)}(X_i(\omega_i),\cdot)\right)\right\|_{L^2(\mathcal{X},\beta)}^2\end{equation}

\begin{equation}\label{onemin}\displaystyle =\frac{1}{N^2}\sum_{j=\min(n,n_N(\omega))}^{\max(n,n_N(\omega))-1}\left\|\sum_{i=1}^N\left(Y_i(\omega_i) K_j(X_i(\omega_i),\cdot)-Y_i(\omega_i) K_{j+1}(X_i(\omega_i),\cdot)\right)\right\|_{L^2(\mathcal{X},\beta)}^2.\end{equation} 

We have used the fact, justified by \(\eqref{eq:perp}\), that the terms in the telescopic sum are orthogonal. It follows that $\displaystyle n\mapsto h_{N,\omega}(n)$ is non-increasing on $\{0,\ldots,n_N(\omega)\}$ and non-decreasing on $\{n_N(\omega),\ldots\}$.

\end{proof}

\begin{thm}
 For almost every $\omega\in\Omega^\mathbb N$ and every $\epsilon>0$, there exists $N_0\in\mathbb N$ such that, for every $N\geq N_0$, there exists a function $h_{N,\omega}:\mathbb N\to\mathbb R$, non-increasing on $\{0,\ldots,n_N(\omega)\}$ and non-decreasing on $\{n_N(\omega),\ldots\}$, with $h_{N,\omega}(n_N(\omega))=1-\|f_m\|_{L^2(\mathcal{X},\beta)}$, such that, for every $n\in\mathbb N$ such that $f_{n,\Sigma_N(\omega)}^*\neq 0$,
$$\displaystyle\left|\mathcal{E}_{1,\rho}\left(\frac{{f_{n,\Sigma_N(\omega)}^*}}{\left\|{f_{n,\Sigma_N(\omega)}^*}\right \|_{L^2(\mathcal{X},\beta)}}\right)-h_{N,\omega}(n)\right|\leq \epsilon.$$

\end{thm}

\begin{proof}

Fix $\epsilon>0$. By Theorem \ref{thm:unifcrmconv} and since $f_m\neq 0$, for a.e. $\omega\in\Omega^\mathbb{N}$, there exists $N_0=N_0(\omega,\epsilon)\in \mathbb{N}$ such that $f^*_{n_N(\omega),\Sigma_N(\omega)}\neq 0$ and $$\displaystyle\left\|\frac{f_m}{\left\|f_m\right\|_{L^2(\mathcal{X},\beta)}}-\frac{f_{n_N(\omega),\Sigma_N(\omega)}^*}{\left\|f_{n_N(\omega),\Sigma_N(\omega)}^*\right\|_{L^2(\mathcal{X},\beta)}}\right\|_{L^2(\mathcal{X},\beta)}\leq \epsilon$$ for $N\geq N_0$. Fix such an $\omega\in\Omega^\mathbb N$. We compute for $n\leq n_N(\omega)$: $$\displaystyle \displaystyle\left\langle \frac{f_{n,\Sigma_N(\omega)}^*}{\left\|f_{n,\Sigma_N(\omega)}^*\right\|_{L^2(\mathcal{X},\beta)}}, \frac{f_{n_N(\omega),\Sigma_N(\omega)}^*}{\left\|f_{n_N(\omega),\Sigma_N(\omega)}^*\right\|_{L^2(\mathcal{X},\beta)}}\right\rangle_{L^2(\mathcal{X},\beta)} =\frac{\left\|f^*_{n,\Sigma_N(\omega)}\right\|_{L^2(\mathcal{X},\beta)}}{\left\|f^*_{n_N(\omega),\Sigma_N(\omega)}\right\|_{L^2(\mathcal{X},\beta)}}.$$
For $n\geq n_N(\omega)$, we obtain in the same way $$\displaystyle \left\langle \frac{f_{n,\Sigma_N(\omega)}^*}{\left\|f_{n,\Sigma_N(\omega)}^*\right\|_{L^2(\mathcal{X},\beta)}}, \frac{f_{n_N(\omega),\Sigma_N(\omega)}^*}{\left\|f_{n_N(\omega),\Sigma_N(\omega)}^*\right\|_{L^2(\mathcal{X},\beta)}}\right\rangle_{L^2(\mathcal{X},\beta)}  =\frac{\left\|f^*_{n_N(\omega),\Sigma_N(\omega)}\right\|_{L^2(\mathcal{X},\beta)}}{\left\|f^*_{n,\Sigma_N(\omega)}\right\|_{L^2(\mathcal{X},\beta)}}.$$ 
Then, Lemma \ref{lem:oneminimum} implies that $$g_{N,\omega}:n\mapsto \left\langle \frac{f_{n,\Sigma_N(\omega)}^*}{\left\|f_{n,\Sigma_N(\omega)}^*\right\|_{L^2(\mathcal{X},\beta)}}, \frac{f_{n_N(\omega),\Sigma_N(\omega)}^*}{\left\|f_{n_N(\omega),\Sigma_N(\omega)}^*\right\|_{L^2(\mathcal{X},\beta)}}\right\rangle_{L^2(\mathcal{X},\beta)} $$ is non-decreasing up to $1$ for $n\leq n_N(\omega)$ and then non-increasing. By the Cauchy--Schwarz inequality, for $N\geq N_0$, $\left|g_{N,\omega}(n)-\frac{v_{N,\omega}(n)}{\|f_m\|_{L^2(\mathcal{X},\beta)}}\right|\leq \epsilon$, for every $n\in \mathbb N$ such that $f_{n,\Sigma_N(\omega)}^*\neq 0$, where $$\displaystyle v_{N,\omega}:n\mapsto \left\langle \frac{f_{n,\Sigma_N(\omega)}^*}{\left\|f_{n,\Sigma_N(\omega)}^*\right\|_{L^2(\mathcal{X},\beta)}}, f_m\right\rangle_{L^2(\mathcal{X},\beta)} .$$ From \eqref{eq:reformulation}, we have $$v_{N,\omega}(n)=\displaystyle1-\mathcal{E}_{1,\rho}\left(\frac{f_{n,\Sigma_N(\omega)}^*}{\left\|f_{n,\Sigma_N(\omega)}^*\right\|_{L^2(\mathcal{X},\beta)}}\right),$$ So $h_{N,\omega}:=1-\|f_m\|_{L^2(\mathcal{X},\beta)}\,g_{N,\omega}$ is non-increasing on $\{0,\ldots,n_N(\omega)\}$, non-decreasing on $\{n_N(\omega),\ldots\}$, satisfies $h_{N,\omega}(n_N(\omega))=1-\|f_m\|_{L^2(\mathcal{X},\beta)}$ and $$\left|\mathcal{E}_{1,\rho}\left(\frac{f_{n,\Sigma_N(\omega)}^*}{\left\|f_{n,\Sigma_N(\omega)}^*\right\|_{L^2(\mathcal{X},\beta)}}\right)-h_{N,\omega}(n)\right|\leq \|f_m\|_{L^2(\mathcal{X},\beta)}\,\epsilon$$ for all $n$ such that $f_{n,\Sigma_N(\omega)}^*\neq 0$. Since $\epsilon>0$ is arbitrary, this proves the statement.\\

\end{proof}


\subsection{Probabilistic interpretation}\label{sub:probabilistic}

In this subsection, still in the setting of binary classification, we consider the probability of misclassification as a risk functional to be minimized. We show how the functions $f^*_{n_N(\omega),\Sigma_N(\omega)}$ provide a solution to this minimization problem.\\

As in Subsection \ref{subsec:statistical}, we consider $X,Y:\Omega\rightarrow\mathcal{X}\times \{-1,1\}$. We need the following notion. The {\it{regression function}} $f_\rho:\mathcal{X}\rightarrow [-1,1]$ is defined as a representative of the conditional expectation $f_\rho(x) := \mathbb{E}(Y|X=x)$. We emphasize the following relation between $f_\rho$ and $f_m$:

\begin{pro}\label{density}We have $\beta$-a.e. \begin{equation}\label{eq:tempered_difference}f_{\rho}\cdot f_X=P(Y=1)f_{X|Y=1}-P(Y=-1)f_{X|Y=-1}=f_m\end{equation}.\end{pro}

\begin{proof}
For $\beta$-almost every $x\in\mathcal X$ such that $f_X(x)>0$, we have

$$
\displaystyle
f_\rho(x)
=\mathbb E(Y\mid X=x)
=\frac{\mu(Y=1)f_{X\mid Y=1}(x)-\mu(Y=-1)f_{X\mid Y=-1}(x)}
{f_X(x)}
=\frac{f_m(x)}{f_X(x)}.
$$

Alternatively, Proposition \ref{lem:alternative} gives, for $\beta$-almost every $x\in\mathcal X$ such that $f_X(x)>0$,

$$
f_\rho(x)=\mathbb E(Y\mid X=x)=\mathbb E(Y_X\mid X=x)=\frac{f_m(x)}{f_X(x)}.
$$

\end{proof}

For the sake of simplicity, given two random variables $Z_1,Z_2\in\mathcal{M}(\Omega, \mathbb{R})$, we denote by $$\displaystyle P(Z_1=Z_2):=\mu\left(\left\{\omega\in\Omega: Z_1(\omega)=Z_2(\omega)\right\}\right),$$ and $P(Z_1\neq Z_2):=1-P(Z_1=Z_2)$, and we use the same notation for conditional probabilities. Define $g:\mathbb{R}^2\rightarrow\mathbb{R}$ by $g(x,y):=\mathbb{I}_{\mathbb{R}^-}(xy)$. For all $f:\mathcal{X}\rightarrow\{-1,1\}$, we can express, using the notation of \eqref{eq:riskfunctional}:

\begin{equation}\label{classif}\displaystyle \mathcal{E}_{g,\rho}(f)=\int_{\mathcal{X}\times\mathcal{Y}}\mathbb{I}_{\mathbb{R}^-}(y f(x))d\rho=P(Y\neq f(X)).\end{equation}

\begin{pro}
For $Y'\in\mathcal{M}(\Omega, \{-1,1\},\sigma(X))$, the quantity $P(Y=Y')$ is maximized by $Y'=\operatorname{sign} f_m(X).$\\
Moreover, let $(X,Y):\Omega\rightarrow \mathcal{X}\times\{-1,1\}$ be a random variable independent of the observations in the random sample $\Sigma_N=((X_1,Y_1),\ldots,(X_N,Y_N))$, but with the same distribution. For $\omega\in \Omega^\mathbb{N}$, let $Y_{n,N}:=\operatorname{sign} f^*_{n,\Sigma_N(\omega)}(X).$ Then, for almost every $\omega\in\Omega^\mathbb{N}$,

$$\displaystyle P(Y=Y_{n_N(\omega),N}\mid \Sigma_N)-\sup_{n\in\mathbb{N}}P(Y=Y_{n,N}\mid \Sigma_N)\underset{N\rightarrow +\infty}{\rightarrow} 0.$$ 
\end{pro}

\begin{proof}

Let $Y'\in\mathcal{M}(\Omega, \{-1,1\},\sigma(X))$. We can write $Y':={g}(X)$ for some Borel function $g:\mathcal{X}\rightarrow \{-1,1\}$ by the Doob–Dynkin Lemma (see e.g. Proposition 3, Chapter 2 in \cite{rao2005}). Then $\beta$-a.e. we have

 $$\displaystyle P(Y=Y'|X=x)=P\left(\frac{YY'+1}{2}=1\mid X=x\right)=E\left(\frac{YY'+1}{2}|X=x\right)$$
$$\displaystyle =\frac{E(Y  (g(X))\mid X=x)+1}{2}=\frac{E(Y|X=x)(g(x))+1}{2}=\frac{f_\rho(x) (g(x))+1}{2}.$$
We see that this expression is maximal for $g=\operatorname{sign} f_\rho = \operatorname{sign} f_m$ on $f_X>0$. (We can adopt here the convention $\operatorname{sign}(0)\in \{-1,1\}$.) Now consider the second statement. By independence of $(X,Y)$ and $\Sigma_N$,
$$\displaystyle P(Y=Y_{n,N}\mid \Sigma_N)=\frac{E(YY_{n,N}+1\mid \Sigma_N)}{2}=\frac{E(E(YY_{n,N}+1|X,\Sigma_N)\mid \Sigma_N)}{2}$$

$$\displaystyle =\frac{E(f_\rho(X)\operatorname{sign} f^*_{n,\Sigma_N(\omega)}(X)\mid \Sigma_N)+1}{2} = \frac{\left\langle f_m,\operatorname{sign} f^*_{n,\Sigma_N(\omega)}\right\rangle_{L^2(\mathcal{X},\beta)} + 1}{2} $$

\begin{equation}\label{eq:encart}\displaystyle= \frac{\left\langle f^*_{n_N(\omega),\Sigma_N(\omega)},\operatorname{sign} f^*_{n,\Sigma_N(\omega)}\right\rangle}{2} + \frac{\left\langle f_m - f^*_{n_N(\omega),\Sigma_N(\omega)},\operatorname{sign} f^*_{n,\Sigma_N(\omega)}\right\rangle}{2} + \frac{1}{2}.
\end{equation}
Fix $\epsilon>0$. By Theorem \ref{thm:unifcrmconv}, for almost every $\omega\in\Omega^{\mathbb{N}}$, for all sufficiently large $N$,
\[
\left\|f_m-f^*_{n_N,\Sigma_N(\omega)}\right\|_{L^2(\mathcal{X},\beta)}<2\epsilon.
\]
Since $\beta$ is a probability measure and
$\|\operatorname{sign}f^*_{n,\Sigma_N}\|_{L^2(\mathcal{X},\beta)}=1$,
the Cauchy--Schwarz inequality yields, uniformly in $n\in\mathbb{N}$,
\[
\left|
P(Y=Y_{n,N}\mid\Sigma_N)
-\frac{\left\langle f^*_{n_N,\Sigma_N(\omega)},
\operatorname{sign}f^*_{n,\Sigma_N(\omega)}\right\rangle_{L^2(\mathcal{X},\beta)}+1}{2}
\right|<\epsilon.
\]
The conclusion follows from the fact that the quantity $$\displaystyle {\left\langle f^*_{n_N(\omega),\Sigma_N(\omega)},\operatorname{sign} f^*_{n,\Sigma_N(\omega)}\right\rangle_{L^2(\mathcal{X},\beta)} }$$ is maximal for $n=n_N(\omega)$. Indeed, $u\cdot\operatorname{sign}(v)\leq |u|$ for all real numbers $u,v$, so that, pointwise integration gives $$
\begin{aligned}
\left\langle f^*_{n_N(\omega),\Sigma_N(\omega)},\operatorname{sign} f^*_{n,\Sigma_N(\omega)}\right\rangle_{L^2(\mathcal X,\beta)}\leq &\left\|f^*_{n_N(\omega),\Sigma_N(\omega)}\right\|_{L^1(\mathcal{X},\beta)}\\
\\ 
=&\left\langle f^*_{n_N(\omega),\Sigma_N(\omega)},\operatorname{sign} f^*_{n_N(\omega),\Sigma_N(\omega)}\right\rangle_{L^2(\mathcal X,\beta)}.
\end{aligned}
$$
Therefore,
\[
-2\epsilon
\leq P(Y=Y_{n_N,N}\mid\Sigma_N)
-\sup_{n\in\mathbb{N}}P(Y=Y_{n,N}\mid\Sigma_N)\leq 0
\]
for all sufficiently large $N$, almost surely. Since $\epsilon>0$ is arbitrary, the claimed convergence follows.

\end{proof}

\noindent From \eqref{eq:encart} we can give the following precision:

\begin{cor}
For a.e. $\omega\in\Omega^\mathbb{N}$, the sequence $\displaystyle\left(\operatorname{sign} f^*_{n_N(\omega),\Sigma_N(\omega)}\right)_N$ is a minimizing sequence for $\mathcal{E}_{g,\rho}$ over the measurable functions $f:\mathcal{X}\rightarrow\{-1,1\}$, that is, $$\displaystyle \mathcal{E}_{g,\rho}\left(\operatorname{sign} f^*_{n_N(\omega),\Sigma_N(\omega)}\right)\underset{N\rightarrow+\infty}{\longrightarrow} P\left(Y\neq \operatorname{sign} f_m(X)\right)=\inf_f \mathcal{E}_{g,\rho}(f).$$
\end{cor}

\begin{proof}
By the computation in the previous proof, for every measurable $f:\mathcal{X}\rightarrow\{-1,1\}$, $$\displaystyle 1-\mathcal{E}_{g,\rho}(f)=P(Y=f(X))=\frac{\left\langle f_m,f\right\rangle_{L^2(\mathcal{X},\beta)}+1}{2}\leq \frac{\|f_m\|_{L^1(\mathcal{X},\beta)}+1}{2}=P(Y=\operatorname{sign}f_m(X)),$$ so that $\displaystyle\inf_f\mathcal{E}_{g,\rho}(f)=\frac{1-\|f_m\|_{L^1(\mathcal{X},\beta)}}{2}$. On the other hand, writing $f^*_N:=f^*_{n_N(\omega),\Sigma_N(\omega)}$, the Cauchy--Schwarz inequality and $\beta(\mathcal{X})=1$ give $$\displaystyle\left|\left\langle f_m,\operatorname{sign} f^*_N\right\rangle_{L^2(\mathcal{X},\beta)}-\|f^*_N\|_{L^1(\mathcal{X},\beta)}\right|=\left|\left\langle f_m-f^*_N,\operatorname{sign} f^*_N\right\rangle_{L^2(\mathcal{X},\beta)}\right|\leq \|f_m-f^*_N\|_{L^2(\mathcal{X},\beta)}$$ and $$\displaystyle\left|\|f^*_N\|_{L^1(\mathcal{X},\beta)}-\|f_m\|_{L^1(\mathcal{X},\beta)}\right|\leq \|f_m-f^*_N\|_{L^2(\mathcal{X},\beta)}.$$ By Theorem \ref{thm:unifcrmconv}, the right-hand sides tend to $0$ for almost every $\omega$, hence $$\displaystyle 1-\mathcal{E}_{g,\rho}(\operatorname{sign} f^*_N)=\frac{\left\langle f_m,\operatorname{sign} f^*_N\right\rangle_{L^2(\mathcal{X},\beta)}+1}{2}\rightarrow \frac{\|f_m\|_{L^1(\mathcal{X},\beta)}+1}{2}.$$
\end{proof}

\end{document}